\documentclass[final]{siamltex}

\usepackage{a4wide,amssymb,euscript,stmaryrd,epsfig}

\usepackage{amssymb,amsmath}
\usepackage{epsfig}
\usepackage{amsfonts}
\usepackage{latexsym}
\usepackage{xspace}
\usepackage{graphicx}
\usepackage{subfigure}
\usepackage{color}

\newtheorem{remark}[theorem]{Remark}

\newtheorem{examples}[theorem]{Examples}

\title{Subspace correction methods for total variation and $\ell_1-$minimization}

\author{Massimo Fornasier\thanks{Johann Radon Institute for Computational and Applied Mathematics (RICAM),
Austrian Academy of Sciences, Altenbergerstrasse 69, A-4040, Linz, Austria Email: {\tt massimo.fornasier@oeaw.ac.at}} \and Carola-Bibiane Sch\"onlieb\thanks{Department of Applied Mathematics and Theoretical Physics (DAMTP),
Centre for Mathematical Sciences,
Wilberforce Road,
Cambridge CB3 0WA,
United Kingdom.Email: {\tt c.b.s.schonlieb@damtp.cam.ac.uk} } }

\begin{document}

\maketitle

\begin{abstract}
This paper is concerned with the numerical minimization of energy functionals in Hilbert spaces involving convex constraints coinciding with a semi-norm for a subspace.
The optimization is realized by alternating minimizations of the functional on a sequence of orthogonal subspaces. On each subspace an iterative proximity-map algorithm is implemented via \emph{oblique thresholding}, which is the main new tool introduced in this work. We provide convergence conditions for the algorithm in order to compute minimizers of the target energy. Analogous results are derived for a parallel variant of the algorithm. Applications are presented in domain decomposition methods for singular elliptic PDE's arising in total variation minimization and in accelerated sparse recovery algorithms based on $\ell_1$-minimization. We include numerical examples which show efficient solutions to classical problems in signal and image processing.
\end{abstract}



\pagestyle{myheadings}
\thispagestyle{plain}
\markboth{M. FORNASIER AND C.-B. SCH\"ONLIEB}{SUBSPACE CORRECTION METHODS FOR OPTIMIZATION}

\section{Introduction}
Let $\mathcal H$ be a real separable Hilbert space. We are interested in the numerical minimization in $\mathcal H$ of the general form of functionals
$$
\mathcal J(u) := \| T u - g \|_{\mathcal H}^2 + 2 \alpha \psi(u),
$$
where $T \in \mathcal L(\mathcal H)$ is a bounded linear operator, $g \in \mathcal H$ is a datum, and $\alpha>0$ is a fixed constant.
The function $\psi:\mathcal H \to \mathbb{R}_+ \cup\{+\infty\}$ is a semi-norm for a suitable subspace $\mathcal H^\psi$ of $\mathcal H$. In particular, we investigate splittings into arbitrary orthogonal subspaces $\mathcal H = V_1 \oplus V_2$ for which we may have
$$
\psi(\pi_{V_1}(u) + \pi_{V_2}(v)) \neq \psi(\pi_{V_1}(u)) + \psi(\pi_{V_2}(v)), \quad u,v \in \mathcal H,
$$
where $\pi_{V_i}$ is the orthogonal projection onto $V_i$.
With this splitting we want to minimize $\mathcal J$ by suitable instances of the following alternating algorithm:  Pick an initial $V_1\oplus V_2 \ni  u_1^{(0)}+ u_2^{(0)} : = u^{(0)} \in \mathcal H^\Psi$, for example $u^{(0)}=0$, and iterate
$$
\left \{ 
\begin{array}{ll}
u_1^{(n+1)} \approx \argmin_{v_1 \in V_1}  \mathcal  J(v_1 +u_2^{(n)}) &\\
u_2^{(n+1)} \approx  \argmin_{v_2 \in V_2} \mathcal J(u_1^{(n+1)} + v_2) &\\
u^{(n+1)}:=u_1^{(n+1)} + u_2^{(n+1)}.
\end{array}
\right.
$$
 This algorithm is implemented by solving the subspace minimizations via an {\it oblique thresholding} iteration.
We provide a detailed analysis of the convergence properties of this sequential algorithm and of its modification for parallel computation.
We motivate this rather general approach by two relevant applications in domain decomposition methods for total variation minimization and in accelerated sparse recovery algorithms based on $\ell_1$-minimization. Nevertheless, the applicability of our results reaches far beyond these particular examples.

\subsection{Domain decomposition methods for singular elliptic PDE's}

Domain decomposition methods were introduced as techniques for solving partial differential equations based on a decomposition of the spatial domain of the problem into
several subdomains \cite{Li88,BrPaWaXu91,Xu92,ChaMat94,QuVa99,XuZi00,LeeXuZi03,BDHP03,NaSzy05}. The initial equation restricted to the subdomains defines a sequence of new local problems. The main goal is to solve the initial equation via the solution of the local problems. This procedure induces a dimension reduction which is the major responsible of the success of such a method.
Indeed, one of the principal motivations  is the formulation of solvers which can be easily parallelized. 
\\
We apply the theory and the algorithms developed in this paper to adapt domain decompositions to the minimization of  functionals with total variation constraints.
Differently from situations classically encountered in domain decomposition methods for nonsingular PDE's, where solutions are usually supposed at least continuous, in our case the interesting solutions may be discontinuous, e.g., along curves in 2D. These discontinuities may cross the interfaces of the domain decomposition patches. Hence, the crucial difficulty is the correct treatment of interfaces, with the preservation of crossing discontinuities and the correct matching where the solution is continuous instead.  We consider the minimization of the functional $\mathcal J$ in the following setting: Let $\Omega\subset\mathbb{R}^d$, for $d=1,2$, be a bounded open set with Lipschitz boundary. We are interested in the case when $\mathcal{H}=L^2(\Omega)$, $\mathcal{H}^\psi=BV(\Omega)$ and $\psi(u)=|D u|(\Omega)$, the variation of $u$. Then a domain decomposition $\Omega=\Omega_1\cup\Omega_2$  induces the space splitting into $V_i :=\{ u \in L^2(\Omega) : \textrm{supp}(u) \subset \Omega_i \}$, $i=1,2$. Hence, by means of the proposed alternating algorithm, we want to minimize  the functional
$$
\mathcal J(u) := \| T u - g \|_{L^2(\Omega)}^2 + 2 \alpha |Du|(\Omega).
$$
The minimization of energies with total variation constraints traces back to the first uses of such a functional model in noise removal in digital images as proposed by Rudin, Osher, and Fatemi \cite{ROF}. There the operator $T$ is just the identity. Extensions to more general operators $T$ and numerical methods for the minimization of the functional appeared later in several important contributions \cite{ChL,DV97, AV97,Ve01,Ch}. From these pioneering and very successful results, the scientific output related to total variation minimization  and its applications in signal and image processing increased dramatically in the last decade. It is not worth here to mention all the possible directions and contributions.
We limit ourself to mention that, to our knowledge, this paper is the first in presenting a successful domain decomposition approach to total variation minimization. The motivation is that several approaches are directed to the solution of the Euler-Lagrange equations associated to the functional $\mathcal J$, which determine a singular elliptic PDE involving the $1$-Laplace operator. 
Due to the fact that $|Du|(\Omega)$ is not differentiable, one has to discretize its subdifferential, and its characterization is indeed hard to implement numerically in a correct way. The lack of a simple characterization of the subdifferential of the total variation especially raises significant difficulties in dealing with discontinuous interfaces between patches of a domain decomposition. Our approach overcomes these difficulties by minimizing the functional via an iterative proximity-map algorithm, as proposed, e.g., in \cite{Ch}, instead of attempting the direct solution of the Euler-Lagrange equations.
It is also worth to mention that, due to the generality of our setting, our approach can be extended to more general subspace decompositions, not only those arising from a domain splitting.
This can open room to more sophisticated multiscale algorithms where $V_i$ are multilevel spaces, e.g., from a wavelet decomposition.

\subsection{Accelerated sparse recovery algorithms based on $\ell_1$-minimization}

In this application, we are concerned with the use of the alternating algorithm  to the case where $\Lambda$ is a countable index set, $\mathcal H = \ell_2(\Lambda)$, and $\psi(u) = \| u\|_{\ell_{1}(\Lambda)}:= \sum_{\lambda \in \Lambda} | u_\lambda|$.  The minimization of the functional 
$$
\mathcal J(u) := \| T u - g \|_{\ell_2(\Lambda)}^2 + 2 \alpha \|u\|_{\ell_1},
$$
proved to be an extremely efficient alternative to the well-known Tikhonov regularization \cite{enhane96}, whenever 
$$
T u = g,
$$
is an ill-posed problem and the solution $u$ is expected to be a vector with a moderate number of nonzero entries.  Indeed, the imposition of the $\ell_1$-constraint does promote a sparse solution. The use of the $\ell_1$ norm as a sparsity-promoting functional can be found first in
reflection seismology and in  deconvolution of seismic traces \cite{CM,SS,TBM}.
In the last decade more understanding of the deep motivations why $\ell_1$-minimization tends to promote sparse recovery was developed. 
Rigorous results began to appear in the late-1980's, with Donoho and Stark \cite{dost89} and Donoho and Logan \cite{dolo92}. Applications for $\ell_1$ minimization in statistical estimation began in the mid-1990's with the introduction of the LASSO algorithm \cite{ti96} (iterative thresholding). In the signal processing community,
Basis Pursuit \cite{chdo98} was proposed in compression applications for extracting a sparse signal representation from highly overcomplete dictionaries.
From these early steps the applications and understanding of $\ell_1$ minimization have continued to
increase dramatically. It is now hard to trace all the relevant results and applications and it is beyond the scope of this paper.  We shall address the interested reader to the review papers \cite{ba,cand}\footnote{The reader can also find a sufficiently comprehensive collection of the ongoing recent developments at the web-site http://www.dsp.ece.rice.edu/cs/.}. We may simply emphasize the importance of the study of $\ell_1$-minimization by saying that, due to the surprisingly effectiveness in several applications, it can be considered today as the ``modern successor'' of least squares. From this lapidary statement it follows the clear need for efficient algorithms for the minimization of $\mathcal J$. An iterative thresholding algorithm was proposed for this task \cite{CW,dadede04,dateve06,SCD, ti96}. We refer also to the recent developments \cite{fora06-1,fora07-1}. Unfortunately, despite its simplicity which makes it very attractive to users, this algorithm does not perform very well.
For this reason, together with other acceleration methods, e.g., \cite{dafolo07}, a ``domain decomposition'' algorithm  was proposed in \cite{fo07}, and we proved its effectiveness in accelerating the convergence and we provided its parallelization. There the domain is the label set $\Lambda$ which is disjointly decomposed $\Lambda = \Lambda_1 \cup \Lambda_2$. This decomposition induces an orthogonal splitting of $\ell_2(\Lambda)$ into the subspaces  $V_i = \ell_2^{\Lambda_i}(\Lambda) :=\{u \in \ell_2(\Lambda): \supp(u) \subset \Lambda_i\}$, $i=1,2$. In this paper we investigate the application of the alternating algorithm to more general orthogonal subspace decompositions and we discuss how the choice can influence convergence properties and speed-up. Again the generality of our approach allows to experiment several possible decompositions, but we limit ourself to present some key numerical examples in specific cases which help to highlight the properties, i.e., virtues and limitations, of the algorithm.

\subsection{Content of the paper}

In section 2 we illustrate the general assumptions on the convex constraint function $\psi$ and the subspace decompositions. 
In section 3 we formulate the minimization problem and motivate the use of the alternating subspace correction algorithm. With section 4 we start the construction of the algorithmic approach to the minimization, introducing the novel concept of \emph{oblique thresholding}, computed via a generalized Lagrange multiplier. In section 5 we investigate convergence properties of the alternating algorithm, presenting sufficient conditions which allow it to converge to minimizers of the target functional $\mathcal J$. The same results are presented in section 6 for a parallel variant of the algorithm. Section 7 is dedicated to applications and numerical experiments in domain decomposition methods for total variation minimization in 1D and 2D problems, and in accelerations of convergence for $\ell_1-$minimization.

\section{Preliminary Assumptions}

We begin this section with a short description of the generic notations used in this paper.

In the following  $\mathcal H$ is a real separable Hilbert space endowed with the norm $\|\cdot\|_{\mathcal H}$. 
For some countable index set $\Lambda$
we denote by $\ell_p=\ell_p(\Lambda)$, $1 \leq p \leq \infty$, 
the space of real sequences $u=(u_\lambda)_{\lambda \in \Lambda}$ 
with norm
\[
\|u\|_p \,=\, \|u\|_{\ell_p} \,:=\, 
\left(\sum_{\lambda \in \Lambda} |u_\lambda|^p\right)^{1/p}, \quad 1\leq p < \infty
\]
and $\|u\|_\infty \,:=\, \sup_{\lambda \in \Lambda} |u_\lambda|$ as usual.
If $(v_\lambda)$ is a sequence of positive weights then we define the weighted
spaces $\ell_{p,v} = \ell_{p,v}(\Lambda) = \{u, (u_\lambda v_\lambda) \in \ell_p(\Lambda)\}$
with norm
\[
\|u\|_{p,v} \,=\, \|u\|_{\ell_{p,v}} \,=\, \|(u_\lambda v_\lambda)\|_p \,=\,
\left(\sum_{\lambda \in \Lambda} v_\lambda^p |u_\lambda|^p)\right)^{1/p}
\]
(with the standard modification for $p=\infty$).
The Euclidean space is denoted by $\mathbb R ^M$ endowed with the Euclidean norm, but we will also use the $M$-dimensional
space $\ell_q^M$, i.e., $\mathbb R ^M$ endowed with the $\ell_q$-norm. By $\mathbb R_+$ we denote the non-negative
real numbers. 

Usually $\Omega \subset \mathbb R^d$ will denote an open bounded set with Lipschitz boundary. The symbol $L^p(\Omega)$ denotes the usual Lebesgue space of  $p$-summable functions, $C^k(\Omega)$ is the space of functions $k$-times continuously differentiable, and $BV(\Omega)$ the space of functions with bounded variation. For a topological vector space $V$ we denote $V'$ its topological dual. Depending on the context, the symbol $\simeq$ may define an equivalence of norms or an isomorphism of spaces or sets. The symbol $1_\Omega$ denotes the characteristic function of the set $\Omega$.

More specific notations will be defined in the paper, where they turn out to be useful.
\subsection{The convex constraint function $\psi$}

 We are given a function $\psi:\mathcal H \to \mathbb{R}_+ \cup\{+\infty\}$ with the following properties:
\begin{itemize}
\item [($\Psi 1$)] $\psi(0)=0$;
\item [($\Psi 2$)] $\psi$ is sublinear, i.e., $\psi(u+v) \leq \psi(u) + \psi(v)$ for all $u,v \in \mathcal H$;
\item [($\Psi 3$)] $\psi$ is 1-homogeneous, i.e., $\psi(\lambda u) = |\lambda| \psi (u)$ for all $\lambda \in \mathbb R$.
\item [($\Psi 4$)] $\psi$ is lower-semincontinuous in $\mathcal H$, i.e., for any converging sequence $u_n \to u$ in $\mathcal H$
$$
\psi(u) \leq \liminf_{n\in \mathbb N} \psi(u_n).
$$
\end{itemize}
Associated with $\psi$ we assume that there exists a dense subspace $\mathcal H^\psi \subset \mathcal H$ for which $\psi|_{\mathcal H^\psi}$ is a seminorm and $\mathcal H^\psi$ endowed with the norm
$$
\| u\|_{\mathcal H^\psi}:= \| u\|_{\mathcal H} + \psi(u),
$$
is a Banach space.
We do not assume instead that $\mathcal H^\psi$ is reflexive in general; note that due to the dense embedding $\mathcal H^\psi \subset \mathcal H$ we have 
$$
\mathcal H^\psi \subset \mathcal H \simeq  \mathcal H' \subset (\mathcal H^\psi)',
$$
and the duality $\langle \cdot, \cdot \rangle_{(\mathcal H^\psi)' \times \mathcal H^\psi}$ extends the scalar product on $\mathcal H$. In particular, $\mathcal H$ is weakly-$*$-dense in $(\mathcal H^\psi)'$. In the following we require
\begin{itemize}
\item [($H1$)] bounded subsets in $\mathcal H^\psi$ are sequentially bounded in another topology $\tau^\psi$ of  $\mathcal H^\psi$;
\item [($H2$)] $\psi$ is lower-semicontinuous with respect to the topology $\tau^\psi$;
\end{itemize}
In practice, we will always require also that 
\begin{itemize}
\item [($H3$)] $\mathcal H^\psi = \{ u \in \mathcal H: \psi(u) < \infty \}$.
\end{itemize}
We list in the following the specific examples we consider in this paper.
\begin{examples}
\label{ex1}
1. Let $\Omega \subset \mathbb R^d$, for $d=1,2$ be a bounded open set  with Lipschitz boundary, and $\mathcal H = L^2 (\Omega)$. We recall that for $u \in L_{loc}^1(\Omega)$
$$
V(u,\Omega) := \sup \left  \{ \int_\Omega u \dv \varphi~dx: \varphi \in \left [ C^1_c(\Omega) \right ]^d, \| \varphi \|_\infty \leq 1 \right \}
$$
is the variation of $u$ and that $u \in BV(\Omega)$ (the space of bounded variation functions, \cite{AFP,EvGa}) if and only if $V(u,\Omega) < \infty$, see \cite[Proposition 3.6]{AFP}. In such a case, $|D(u)|(\Omega) =V(u,\Omega)$, where $|D(u)|(\Omega)$ is the total variation of the finite Radon measure $D u$, the derivative of $u$ in the sense of distributions. Thus, we define $\psi(u) = V(u,\Omega)$ and it is immediate to see that $\mathcal H^\psi$ must coincide with $ BV(\Omega)$. Due to the embedding $L^2(\Omega) \subset L^1(\Omega)$ and the Sobolev embedding \cite[Theorem 3.47]{AFP} we have
$$
\| u\|_{\mathcal H^\psi} = \| u\|_2 + V(u,\Omega) \simeq \| u\|_1 + |D u|(\Omega) = \| u\|_{BV}.
$$
Hence $(\mathcal H^\psi, \| \cdot \|_{\mathcal H^\psi})$ is indeed a Banach space. It is known that $V(\cdot,\Omega)$ is lower-semincontinuous with respect to $L^2(\Omega)$ \cite[Proposition 3.6]{AFP}. We say that a sequence $(u_n)_n$ in $BV(\Omega)$ converges to $u \in BV(\Omega)$ with the weak-$*$-topology if $(u_n)_n$ converges to $u$ in $L^1(\Omega)$ and $D u_n$ converges to $D u$ with the weak-$*$-topology in the sense of the finite Randon measures. Bounded sets in $BV(\Omega)$ are sequentially weakly-$*$-compact (\cite[Proposition 3.13]{AFP}), and $V(\cdot,\Omega)$ is lower-semicontinuous with respect to the weak-$*$-topology. 

2. Let $\Lambda$ be a countable index set and $\mathcal H = \ell_2(\Lambda)$. For a strictly positive sequence $w = (w_\lambda)_{\lambda \in \Lambda}$, i.e., $w_\lambda \geq w_0 >0$, we define $\psi(u) = \| u\|_{\ell_{1,w}(\Lambda)}:= \sum_{\lambda \in \Lambda} w_\lambda | u_\lambda|$. The space $\mathcal H^\psi$ simply coincides with $\ell_{1,w}(\Lambda)$. Observe that bounded sets in $\mathcal H^\psi$ are sequentially weakly compact in $\mathcal H$ and that, by Fatou's lemma, $\psi$ is lower-semicontinuous with respect to both strong and weak topologies of $\mathcal H$. 

3. Let $\mathcal H = \mathbb R^N$ endowed with the Euclidean norm, and $Q:\mathbb R^N \to \mathbb R^n$, for $n\leq N$, is a fixed linear operator. We define $\psi(u) = \| Q u \|_{\ell_1^n}$. Clearly $\mathcal H^\psi = \mathbb R^N$ and all the requested properties are trivially fulfilled. One particular example of this finite dimensional situation is associated with the choice of $Q:\mathbb R^N \to \mathbb R^{N-1}$ given by $Q(u)_i := N(u_{i+1} - u_i)$, $i=0,\dots,N-2$. In this case $\psi(u) = \| Q u \|_{\ell_1^{N-1}}$ is the discrete variation of the vector $u$ and the model can be seen as a discrete approximation to the situation encountered in the first example, by discrete sampling and finite differences, i.e., setting $u_i := u(\frac{i}{N})$ and $u \in BV(0,1)$.
\end{examples}

\subsection{Bounded subspace decompositions}
\label{projbnd}
In the following we will consider orthogonal decompositions of $\mathcal H$ into closed subspaces. We will also require that such a splitting is bounded in $\mathcal H^\psi$.
\\
Assume that $V_1,V_2$ are two closed, mutually orthogonal, and complementary subspaces of $\mathcal H$, i.e., $\mathcal H = V_1 \oplus V_2$, and $\pi_{V_i}$ are the corresponding orthogonal projections, for $i=1,2$. Moreover we require the mapping property 
$$
\pi_{V_i}|_{\mathcal H^\psi}:\mathcal H^\psi \to V_i^\psi:= \mathcal H^\psi \cap V_i, \quad i=1,2,
$$
continuously in the norm of $\mathcal H^\psi$, and $\Range(\pi_{V_i}|_{\mathcal H^\psi})=V_i^\psi$ is closed. This implies that $\mathcal H^\psi$ splits into the direct sum $\mathcal H^\psi = V_1^\psi \oplus  V_2^\psi$. 
\begin{examples}
\label{ex2}
1. Let $\Omega_1 \subset \Omega \subset \mathbb R^d$, for $d=1,2$, be two bounded open sets with Lipschitz boundaries, and $\Omega_2 = \Omega \setminus \Omega_1$. We define 
$$
V_i :=\{ u \in L^2(\Omega) : \supp(u) \subset \Omega_i \}, \quad i=1,2.
$$
Then $\pi_{V_i} (u ) = u 1_{\Omega_i}$.
For $\psi(u)=V(u,\Omega)$, by \cite[Corollary 3.89]{AFP}, $V_i^\psi = BV(\Omega) \cap V_i$ is a closed subspace of $BV(\Omega)$ and $\pi_{V_i} (u )=u 1_{\Omega_i} \in V_i^\psi$, $i=1,2$, for all $u \in BV(\Omega)$. 
\\
2. Similar decompositions can be considered for the examples where $\mathcal H = \ell_2(\Lambda)$ and $\psi (u) = \|u\|_{\ell_{1,w}}$, see, e.g., \cite{fo07}, and $\mathcal H = \mathbb R^N$ and $\psi (u) = \|Q u\|_{\ell_{1}^n}$.

\end{examples}

\section{A Convex Variational Problem and Subspace Splitting}

We are interested in the minimization in $\mathcal H$ (actually in $\mathcal H^\psi$) of the functional
$$
\mathcal J(u) := \| T u - g \|_{\mathcal H}^2 + 2 \alpha \psi(u),
$$
where $T \in \mathcal L(\mathcal H)$ is a bounded linear operator, $g \in \mathcal H$ is a datum, and $\alpha>0$ is a fixed constant.
 In order to guarantee the existence of its minimizers we assume that:
\begin{itemize}
\item[(C)] $\mathcal J$ is coercive in $\mathcal H$, i.e., $\{ \mathcal J \leq C\} := \{ u \in \mathcal H :  \mathcal J(u) \leq C\}$ is bounded in $\mathcal H$.
\end{itemize}
\begin{examples}
\label{exa3}
1. Assume $\Omega \subset \mathbb R^d$, for $d=1,2$ be a bounded open set with Lipschitz boundary, $\mathcal H = L^2 (\Omega)$ and $\psi(u) = V(u,\Omega)$ (compare Examples \ref{ex1}.1).
In this case we deal with total variation minimization. It is well-known that if $T 1_{\Omega} \neq 0$ then condition (C) is indeed satisfied, see \cite[Proposition 3.1]{Ve01} and \cite{ChL}.

2.  Let $\Lambda$ be a countable index set and $\mathcal H = \ell_2(\Lambda)$. For a strictly positive sequence $w = (w_\lambda)_{\lambda \in \Lambda}$, i.e., $w_\lambda \geq w_0 >0$, we define $\psi(u) = \| u\|_{\ell_{1,w}(\Lambda)}:= \sum_{\lambda \in \Lambda} w_\lambda | u_\lambda|$ (compare with Examples \ref{ex1}.2). In this case condition (C) is trivially satisfied since $\mathcal J(u) \geq 2 \alpha \psi(u) =   2 \alpha \| u\|_{\ell_{1,w}(\Lambda)} \geq \gamma \| u\|_{\ell_2(\Lambda)}$, for $\gamma = 2 \alpha w_0>0$.
\end{examples}

\begin{lemma}
\label{exmin}
Under the assumptions above, $\mathcal J$ has minimizers in $\mathcal H^\psi$.
\end{lemma}
\begin{proof} The proof is a standard application of the direct method of calculus of variations.
Let $(u_n)_n \subset \mathcal H$, a minimizing sequence. By assumption (C) we have $\| u_n\|_{\mathcal H} + \psi(u_n) \leq
 C$ for all $n \in \mathbb N$. Therefore by (H1) we can extract a subsequence in $\mathcal H^\psi$ converging in the topology $\tau^\psi$. Possibly passing to a further subsequence we can assume that it also converges weakly in $\mathcal H$. By lower-semicontinuity of $\| T u - g \|_{\mathcal H}^2$ with respect to the weak topology of $\mathcal H$ and the lower-semicontinuity of $\psi$ with respect to the topology $\tau^\psi$, ensured by assumption (H2), we have the wanted existence of minimizers.
\end{proof}

The minimization of $\mathcal J$ is a classical problem \cite{ET} which was recently re-considered by several authors, \cite{Ch,CW,dadede04,dateve06,SCD, ti96},  with emphasis on the computability of minimizers  in particular cases. They studied essentially the same algorithm for the minimization. \\
For $\psi$ with properties $(\Psi1-\Psi4)$, there exists a closed convex set $K_\psi \subset \mathcal H$ such that
\begin{eqnarray*}
\psi^*(u) &=& \sup_{v \in \mathcal H} \{ \langle v,u \rangle - \psi(v) \}\\
&=& \chi_{K_\psi}(u) = \left \{ \begin{array}{ll}
0 & \mbox{ if } u \in K_\psi\\
+\infty & \mbox{ otherwise}.
\end{array}
\right .
\end{eqnarray*}
See also Examples \ref{ex3}.2 below.
In the following we assume furthermore that $K_\psi = - K_\psi$. For any closed convex set $K \subset \mathcal H$ we denote $P_K(u)= \argmin_{v \in K} \|u-v\|_{\mathcal H}$ the orthogonal projection onto $K$. For $\mathbb S_\alpha^\psi := I - P_{\alpha K_\psi}$, called the {\it generalized thresholding map} in the signal processing literature, the iteration
\begin{equation}
\label{eq1}
u^{(n+1)} = \mathbb S_\alpha^\psi ( u^{(n)} + T^* ( g - T u^{(n)}))
\end{equation}
converges weakly to a minimizer $u \in \mathcal H^\psi$ of $\mathcal J$, for any initial choice $u^{(0)} \in \mathcal H^\psi$, provided $T$ and $g$ are suitably rescaled so that $\|T\|<1$.
For particular situations, e.g., $\mathcal H = \ell_2(\Lambda)$ and $\psi (u) = \|u\|_{\ell_{1,w}}$, one can prove the convergence in norm \cite{dadede04,dateve06}.
\\

As it is pointed out, for example in \cite{dafolo07,fo07}, this algorithm converges with a poor rate, unless $T$ is non-singular or has special additional spectral properties. For this reason accelerations by means of projected steepest descent iterations \cite{dafolo07} and domain decomposition methods \cite{fo07} were proposed. \\

The particular situation considered in  \cite{fo07} is $\mathcal H = \ell_2(\Lambda)$ and $\psi (u) = \|u\|_{\ell_{1}(\Lambda)}$. In this case one takes advantage of the fact that for a disjoint partition of the index set $\Lambda = \Lambda_1 \cup \Lambda_2$ we have the splitting $\psi(u_{\Lambda_1} + u_{\Lambda_2}) = \psi(u_{\Lambda_1})+  \psi(u_{\Lambda_2})$ for any vector $u_{\Lambda_i}$ supported on $\Lambda_i$, $i=1,2$. Thus, a decomposition into column subspaces (i.e., componentwise) of the operator $T$ (if identified with a suitable matrix) is realized, and alternating minimizations on these subspaces are performed by means of iterations of the type (\ref{eq1}). This leads, e.g., to the following sequential algorithm:
 Pick an initial $u_{\Lambda_1}^{(0,L)}+ u_{\Lambda_2}^{(0,M)} : = u^{(0)} \in \ell_1(\Lambda)$, for example $u^{(0)}=0$, and iterate
\begin{equation}
\label{schw_sp:it}
\left \{ 
\begin{array}{ll}
\left \{ 
\begin{array}{ll}
u_{\Lambda_1}^{(n+1,0)} =  u_{\Lambda_1}^{(n,L)}&\\
u_{\Lambda_1}^{(n+1,\ell+1)} =  \mathbb{S}_\alpha \left ( u_{\Lambda_1}^{(n+1,\ell)} +  T_{\Lambda_1}^*((g -  T_{\Lambda_2}u_{\Lambda_2}^{(n,M)})-  T_{\Lambda_1}  u_{\Lambda_1}^{(n+1,\ell)}) \right )& \ell=0,\dots, L-1\\
\end{array}\right. &\\
\left \{ 
\begin{array}{ll}
u_{\Lambda_2}^{(n+1,0)} =  u_{\Lambda_2}^{(n,M)}&\\
u_{\Lambda_2}^{(n+1,\ell+1)} = \mathbb{S}_\alpha \left ( u_{\Lambda_2}^{(n+1,\ell)} + T_{\Lambda_2}^*((g -  T_{\Lambda_1}u_{\Lambda_1}^{(n+1,L)})-  T_{\Lambda_2}  u_{\Lambda_2}^{(n+1,\ell)}) \right )&\ell=0,\dots, M -1\\
\end{array}\right. &\\
u^{(n+1)}:=u_{\Lambda_1}^{(n+1,L)} + u_{\Lambda_2}^{(n+1,M)}.
\end{array}
\right.
\end{equation} 
Here the operator $ \mathbb{S}_\alpha $ is the soft-thresholding operator which acts componentwise $\mathbb{S}_\alpha v = (S_\alpha v_\lambda)_{\lambda \in \Lambda}$ and defined by
\begin{equation}
\label{softthr}
S_\alpha (x ) = \left \{
\begin{array}{ll}
x -\sgn(x)\alpha, & |x| > \alpha\\
0, & \mbox{ otherwise}.
\end{array}
\right.
\end{equation}
The expected benefit from this approach
is twofold:
\begin{itemize}
\item[1.] Instead of solving one large problem with many iteration steps, we can solve approximatively
several smaller subproblems, which might lead to an acceleration of convergence and a reduction of the overall
computational effort, due to possible conditioning improvements;
\item[2.] The subproblems do not need more sophisticated algorithms, simply reproduce at smaller dimension the original problem, and they can be solved in parallel.
\end{itemize}

The nice splitting of $\psi$ as a sum of evaluations on subspaces does not occur, for instance, when $\mathcal H = L^2(\Omega)$, $\psi(u)=V(u,\Omega)=|D u|(\Omega)$, and $\Omega_1 \cup \Omega_2\subset \Omega \subset \bar \Omega_1 \cup \bar \Omega_2$ is a disjoint decomposition of $\Omega$. Indeed, cf. \cite[Theorem 3.84]{AFP}, we have
\begin{equation}
\label{amb}
|D(u_{\Omega_1} + u_{\Omega_2})|(\Omega) = |D u_{\Omega_1}|(\Omega_1) + |D u_{\Omega_2}|(\Omega_2)+\underbrace{\int_{\partial \Omega_1 \cap \partial \Omega_2} |  u_{\Omega_1}^+(x) - u_{\Omega_2}^-(x)| d \mathcal H_1(x)}_{\mbox{additional interface term}}.
\end{equation}
Here one should not confuse $ \mathcal H_d$ with any $ \mathcal H^\psi$ since the former indicates the Hausdorff measure of dimension $d$. 
The symbols $v^+$ and $v^-$ denote the left and right approximated limits at jump points \cite[Proposition 3.69]{AFP}.
The presence of the additional boundary interface term $\int_{\partial \Omega_1 \cap \partial \Omega_2} |  u_{\Omega_1}^+(x) - u_{\Omega_2}^-(x)| d \mathcal H_1(x)$ does not allow to use in a straightforward way iterations as in (\ref{eq1}) to minimize the local problems on $\Omega_i$. 
\\
Moreover, also in the sequence space setting mentioned above, the hope for a better conditioning by column subspace splitting as in \cite{fo07} might be ill-posed, no such splitting needs to be well conditioned in general (good cases are provided in \cite{tr07} instead).\\

Therefore, one may want to consider arbitrary subspace decompositions and, in order to deal with these more general situations, we investigate splittings into arbitrary orthogonal subspaces $\mathcal H = V_1 \oplus V_2$ for which we may have
$$
\psi(\pi_{V_1}(u) + \pi_{V_2}(v)) \neq \psi(\pi_{V_1}(u)) + \psi(\pi_{V_2}(v)).
$$
In principal, in this paper we limit ourself to consider the detailed analysis for two subspaces $V_1, V_2$. Nevertheless, the arguments can be easily generalized to multiple subspaces $V_1, \dots, V_{\mathcal N}$, see, e.g., \cite{fo07}, and in the numerical experiments we will also test this more general situation.\\

With this splitting we want to minimize $\mathcal J$ by suitable instances of the following alternating algorithm:  Pick an initial $V_1\oplus V_2 \ni  u_1^{(0)}+ u_2^{(0)} : = u^{(0)} \in \mathcal H^\Psi$, for example $u^{(0)}=0$, and iterate
\begin{equation}
\label{schw_sp}
\left \{ 
\begin{array}{ll}
u_1^{(n+1)} \approx \argmin_{v_1 \in V_1}  \mathcal  J(v_1 +u_2^{(n)}) &\\
u_2^{(n+1)} \approx  \argmin_{v_2 \in V_2} \mathcal J(u_1^{(n+1)} + v_2) &\\
u^{(n+1)}:=u_1^{(n+1)} + u_2^{(n+1)}.
\end{array}
\right.
\end{equation}
We use ``$\approx$'' (the approximation symbol) because in practice we never perform the exact minimization, as it occurred in (\ref{schw_sp:it}). In the following section we discuss how to realize the approximation to the individual subspace minimizations. As pointed out above, this cannot just reduce to a simple iteration of the type (\ref{eq1}).

\section{Local Minimization by Lagrange Multipliers}

Let us consider, for example,
\begin{equation}
\label{m1}
\argmin_{v_1 \in V_1}  \mathcal  J(v_1 +u_2) = \argmin_{v_1 \in V_1} \| T v_1 - (g - T u_2) \|^2_{\mathcal H} + 2 \alpha \psi(v_1 + u_2).
\end{equation}
First of all, observe that $\{u \in \mathcal H: \pi_{V_2} u = u_2 ,\mathcal J(u) \leq C\} \subset  \{\mathcal J \leq C\}$, hence the former set is also bounded by assumption (C). By the same argument as in Lemma \ref{exmin}, the minimization (\ref{m1}) has solutions. It is useful to us to introduce an auxiliary functional $\mathcal J^s_1$, called the {\it surrogate functional} of $\mathcal J$: Assume $a,u_1 \in V_1$ and $u_2 \in V_2$ and define
\begin{equation}
\label{surrfunc}
\mathcal J^s_1(u_1+ u_2, a) := \mathcal J(u_1+ u_2)+ \| u_1 -a\|_{\mathcal H}^2 - \| T(u_1 -a)\|_{\mathcal H}^2.
\end{equation}
A straightforward computation shows that
$$
\mathcal J^s_1(u_1+ u_2, a) = \| u_1 - (a + \pi_{V_1}T^*( g - T u_2 - T a))\|_{\mathcal H}^2 + 2 \alpha \psi(u_1 + u_2) + \Phi(a,g,u_2),
$$
where $\Phi$ is a function of $a,g,u_2$ only. We want to realize an approximate solution to (\ref{m1}) by using the following algorithm: For $u_1^{(0)} \in V_1^\psi$,
\begin{equation}
\label{m2}
u_1^{(\ell+1)} = \argmin_{u_1 \in V_1}  \mathcal  J^s_1(u_1 +u_2, u_1^{(\ell)}), \quad \ell \geq 0.
\end{equation}
Before proving the convergence of this algorithm, we need to investigate first how to compute practically $u_1^{(n+1)}$ for $u_1^{(n)}$ given. To this end we need to introduce further notions and to recall some useful results.
\subsection{Generalized Lagrange multipliers for nonsmooth objective functions}
Let us begin this subsection with the notion of a subdifferential.
\begin{definition}
\label{subdiff}
For a locally convex space $V$ and for a convex function $F:V \to \mathbb R \cup \{-\infty,+\infty\}$, we define the \emph{subdifferential} of $F$ at $x \in V$, as $\partial F(x) = \emptyset$ if $F(x) = \infty$, otherwise
$$
\partial F(x) := \partial F_V(x)  :=\{x^* \in V' : \langle x^*,y-x\rangle + F(x) \leq F(y) \quad \forall y \in V\},
$$
where $V'$ denotes the dual space of $V$. It is obvious from this
definition that $0 \in \partial F(x)$ if and only if $x$ is a minimizer
of $F$. 
Since we deal with several spaces, namely, $\mathcal H,\mathcal H^\psi, V_i, V_i^\psi$, it will turn out to be useful to distinguish sometimes in which space (and associated topology) the subdifferential is defined by imposing a subscript $\partial_V F$ for the subdifferential considered on the space $V$.
\end{definition}

\begin{examples}
\label{ex3}
1. Let $V=\ell_1(\Lambda)$ and $F(x) := \|x\|_1$ is the $\ell_1-$norm. We have
\begin{equation}\label{subdiff_l1}
\partial \|\cdot\|_1(x) \,=\, \{ \xi \in \ell_\infty(\Lambda):~ \xi_\lambda \in \partial |\cdot|(x_\lambda), \lambda \in \Lambda\}
\end{equation}
where $\partial |\cdot|(z) = \{\sgn(z)\}$ if $z \neq 0$ and $\partial |\cdot|(0) = [-1,1]$.

2. Assume $V= \mathcal H$ and $\varphi \geq 0$ is a proper lower-semicontinuous convex function. For $F(u;z) = \|u-z\|_{\mathcal{H}}^2 + 2 \varphi(u)$, we define the function
$$
\prox_\varphi(z) := \argmin_{u \in V} F(u;z),
$$
which is called the \emph{proximity map} in the convex analysis literature, e.g., \cite{ET,CW}, and \emph{generalized thresholding} in the signal processing literature, e.g., \cite{dadede04,dafolo07,dateve06,fo07}.  
Observe that by $\varphi \geq 0$ the function $F$ is coercive in $\mathcal H$ and by lower-semicontinuity and strict convexity of the term $\|u-z\|_{\mathcal{H}}^2$ this definition is well-posed. 
In particular, $\prox_\varphi(z)$ is the unique solution of the following differential inclusion
$$
0 \in (u-z) + \partial \varphi(u).
$$
It is well-known \cite{ET,rowe98} that the proximity map is nonexpansive, i.e.,
$$
\| \prox_\varphi(z_1) - \prox_\varphi(z_2)\|_{\mathcal H} \leq \| z_1 - z_2\|_{\mathcal H}.
$$ 
In particular, if $\varphi$ is a 1-homogeneous function then
$$
\prox_\varphi(z) = (I-P_{K_\varphi})(z),
$$
where $K_\varphi$ is a suitable closed convex set associated to $\varphi$, see for instance \cite{CW}.
\end{examples}\\ 
Under the notations of Definition \ref{subdiff}, we consider the following problem
\begin{equation}
\label{pbbp}
\argmin_{x \in V} \{ F(x): G(x)=0\},
\end{equation}
where $ G:V \to \mathbb R$ is a bounded linear operator on $V$. We have the following useful result.

\begin{theorem}[Generalized Lagrange multipliers for nonsmooth objective functions, Theorem 1.8, \cite{bapr}]
\label{bpth}
If $F$ is continuous in a point of $\ker G$ and $G$ has closed range in $V$, then a point $x_0 \in \ker G$ is an optimal solution of (\ref{pbbp}) if and only if 
$$
\partial F(x_0) \cap \Range G^* \neq \emptyset.
$$
\end{theorem}

\subsection{Oblique thresholding}

We want to exploit Theorem \ref{bpth} in order to produce an algorithmic solution to each iteration step (\ref{m2}).
\begin{theorem}[Oblique thresholding]
\label{main1}
For $u_2 \in V_2^\psi$ and for $z \in V_1$ the following statements are equivalent:
\begin{itemize}
\item[(i)] $u_1^* = \argmin_{u \in V_1} \| u -z\|_{\mathcal{H}}^2 + 2 \alpha \psi(u+u_2)$;
\item[(ii)] there exists $\eta \in \Range (\pi_{V_2}|_{\mathcal H^\psi})^* \simeq (V_2^\psi)'$ such that
$0 \in u^*_1 -(z- \eta) + \alpha \partial_{\mathcal H^\psi} \psi(u_1^* + u_2)$.
\end{itemize}
Moreover, the following statements are equivalent and imply (i) and (ii).
\begin{itemize}
\item[(iii)] there exists $\eta \in V_2$ such that $u_1^* = (I- P_{\alpha K_\psi})(z+ u_2 -\eta)-u_2=\mathbb S_\alpha^\psi(z+ u_2 -\eta) -u_2 \in V_1$;
\item[(iv)] there exists $\eta \in V_2$ such that
$\eta = \pi_{V_2} P_{\alpha K_\psi} (\eta - (z+ u_2))$.
\end{itemize}
\end{theorem}
$\vspace{0.1cm}$
\begin{proof}
Let us show the equivalence between (i) and (ii). The problem in (i) can be reformulated as
$$
u_1^* = \argmin_{u \in \mathcal H^\psi} \{ F(u):=\| u -z\|_{\mathcal{H}}^2 + 2 \alpha \psi(u+u_2), \pi_{V_2}(u)=0\}.
$$
The latter is a special instance of (\ref{pbbp}).
Moreover, $F$ is continuous on $V_1^\psi \subset  V_1 =\ker \pi_{V_2}$ in the norm-topology of $\mathcal H^\psi$ (while in general it is not on $V_1$ with the norm topology of $\mathcal H$). Recall now that $\pi_{V_2}|_{\mathcal H^\psi}$ is assumed to be a bounded and surjective map with closed range in the norm-topology of $\mathcal H^\psi$ (see Section \ref{projbnd}). This means that $(\pi_{V_2}|_{\mathcal H^\psi})^*$ is injective and that $\Range (\pi_{V_2}|_{\mathcal H^\psi})^* \simeq (V_2^\psi)'$ is closed. Therefore, by an application of  Theorem \ref{bpth} the optimality of $u^*_1$ is equivalent to the existence of $\eta \in \Range (\pi_{V_2}|_{\mathcal H^\psi})^* \simeq (V_2^\psi)'$ such that 
$$
-\eta \in \partial_{\mathcal H^\psi} F(u_1^*).
$$
Due to the continuity of $\|u-z\|_{\mathcal H}^2$ in $\mathcal H^\psi$, we have, by \cite[Proposition 5.6]{ET}, that 
$$
\partial_{\mathcal H^\psi} F(u_1^*) = 2 (u_1^*-z) + 2 \alpha \partial_{\mathcal H^\psi} \psi(u_1^*+u_2).
$$
Thus, the optimality of $u^*_1$ is equivalent to
$$
0 \in u^*_1 -(z- \eta) + \alpha \partial_{\mathcal H^\psi} \psi(u_1^* + u_2).
$$
This concludes the equivalence of (i) and (ii). Let us show now that (iii) implies (ii). 
The condition in (iii) can be rewritten as
$$
\xi = (I- P_{\alpha K_\psi})(z+u_2 - \eta), \quad \xi =u_1^* + u_2. 
$$
Since $\psi \geq 0 $ is 1-homogeneous and lower-semincontinuous, by Examples \ref{ex3}.2, the latter is equivalent to
$$
0 \in \xi - (z+u_2 - \eta) + \alpha \partial_{\mathcal H} \psi(\xi)
$$
or, by (H3),
\begin{eqnarray*}
\xi &=& \argmin_{u \in \mathcal H} \| u - (z+u_2 - \eta)\|_{\mathcal{H}}^2 + 2 \alpha \psi(u)\\
&=& \argmin_{u \in \mathcal H^\psi} \| u - (z+u_2 - \eta)\|_{\mathcal{H}}^2 + 2 \alpha \psi(u)
\end{eqnarray*}
The latter optimal problem is equivalent to
$$
0 \in  \xi - (z+u_2 - \eta) + \alpha \partial_{\mathcal H^\psi} \psi(\xi) \mbox{ or } 0 \in u^*_1 -(z- \eta) + \alpha \partial_{\mathcal H^\psi} \psi(u_1^* + u_2).
$$
Since $V_2 \subset (V_2^\psi)' \simeq\Range (\pi_{V_2}|_{\mathcal H^\psi})^*$ we obtain that (iii) implies (ii). We prove now the equivalence between (iii) and (iv). We have
\begin{eqnarray*}
u_1^* &=& (I- P_{\alpha K_\psi})(z+ u_2 -\eta)-u_2 \in V_1\\
&=& z-\eta - P_{\alpha K_\psi}(z+ u_2 -\eta).
\end{eqnarray*}
By applying $\pi_{V_2}$ to both sides of the latter equality we get
$$
0 = -\eta - \pi_{V_2} P_{\alpha K_\psi}(z+ u_2 -\eta).
$$
By recalling that $K_\psi = - K_\psi$, we obtain the fixed point equation
\begin{equation}
\label{fixpt}
\eta =  \pi_{V_2} P_{\alpha K_\psi}(\eta- (z+ u_2)).
\end{equation}
Conversely, assume $\eta =  \pi_{V_2} P_{\alpha K_\psi}(\eta- (z+ u_2))$ for some $\eta \in V_2$. Then
\begin{eqnarray*}
(I- P_{\alpha K_\psi})(z+ u_2 -\eta) -u_2 &=&  z-\eta - P_{\alpha K_\psi}(z+ u_2 -\eta)\\
&=& z -  \pi_{V_2} P_{\alpha K_\psi}(\eta- (z+ u_2)) -  P_{\alpha K_\psi}(z+ u_2 -\eta)\\
&=& z - (I- \pi_{V_2})  P_{\alpha K_\psi}(z+ u_2 -\eta)\\
&=& z - \pi_{V_1} P_{\alpha K_\psi}(z+ u_2 -\eta) = u_1^* \in V_1.
\end{eqnarray*}
\end{proof}
\begin{remark}
\label{rem1}
1. Unfortunately in general we have $V_2 \subsetneq (V_2^\psi)'$ which excludes the complete equivalence of the previous conditions (i)-(iv).
For example, in the case $\mathcal H = \ell_2(\Lambda)$ and $\psi(u) =\|u\|_{\ell_1}$, $\Lambda = \Lambda_1 \cup \Lambda_2$, $V_i = \ell_2^{\Lambda_i}(\Lambda) :=\{u \in \ell_2(\Lambda): \supp(u) \subset \Lambda_i\}$, $i=1,2$, we have $V_2^\psi = \ell_1^{\Lambda_2}(\Lambda)=\{u \in \ell_1(\Lambda): \supp(u) \subset \Lambda_2\}$, hence, $V_2 \subsetneq (V_2^\psi)' \simeq \ell_\infty^{\Lambda_2}(\Lambda)=\{u \in \ell_\infty(\Lambda): \supp(u) \subset \Lambda_i\}$. It can well be that $\eta \in \ell_\infty^{\Lambda_2}(\Lambda) \setminus \ell_2^{\Lambda_2}(\Lambda)$. However, since $\psi(u_{\Lambda_1} + u_{\Lambda_2}) = \psi(u_{\Lambda_1})+  \psi(u_{\Lambda_2})$ in this case, we have $0 \in u_1^* - z +\alpha \partial \| \cdot \|_1(u_1^*)$ and therefore we may choose any $\eta$ in $\partial \| \cdot \|_1(u_2)$. Following \cite{fo07}, $u_2$ is assumed to be the result of soft-thresholded iterations, hence $u_2$ is a finitely supported vector. Therefore, by Examples \ref{ex3}.1, we can choose $\eta$ to be also a finitely supported vector, hence $\eta \in \ell_2^{\Lambda_2}(\Lambda)=V_2$. This means that the existence of $\eta \in V_2$ as in (iii) or (iv) of the previous theorem may occur also in those cases for which  $V_2 \subsetneq (V_2^\psi)'$. In general, we can only observe that $V_2$ is weakly-$*$-dense in $(V_2^\psi)'$.

2. For $\mathcal H$ with finite dimension -- which is the relevant case in numerical applications -- all the spaces are independent of the particular attached norm and coincide with their duals, hence all the statements (i)-(iv) of the previous theorem are equivalent in this case.
\end{remark}\\

A simple constructive test for the existence of  $\eta \in V_2$ as in (iii) or (iv) of the previous theorem is provided by the following iterative algorithm:
\begin{equation}
\label{fixptit}
\eta^{(0)} \in V_2, \quad \eta^{(m+1)} =  \pi_{V_2} P_{\alpha K_\psi}(\eta^{(m)}- (z+ u_2)), \quad m \geq 0.
\end{equation}
\begin{proposition}\label{fixptitpr}
The following statements are equivalent:
\begin{itemize}
\item[(i)] there exists $\eta \in V_2$ such that
$\eta = \pi_{V_2} P_{\alpha K_\psi} (\eta - (z+ u_2))$ (which is in turn the condition (iv) of Theorem \ref{main1})
\item[(ii)] the iteration (\ref{fixptit}) converges weakly to any $\eta \in V_2$ that satisfies (\ref{fixpt}).
\end{itemize} In particular, there are no fixed points of (\ref{fixpt}) if and only if $\|\eta^{(m)}\|_{\mathcal H} \to \infty$, for $m \to \infty$.
\end{proposition}

For the proof of this Proposition we need to recall some classical notions and results.
\begin{definition}
A nonexpansive map $T:\mathcal H \to \mathcal H$ is strongly nonexpansive if for $(u_n-v_n)_n$ bounded and $\| T(u_n) - T(v_n)\|_{\mathcal H} - \| u_n-v_n\|_{\mathcal H} \to 0$ we have
$$
u_n-v_n - T(u_n) - T(v_n) \to 0, \quad n \to \infty.
$$
\end{definition}
\begin{proposition}[Corollaries 1.3, 1.4, and 1.5 \cite{br}]
\label{bbl}
Let $T:\mathcal H \to \mathcal H$ be a strongly nonexpansive map. Then $\fix T =\{ u \in \mathcal H: T(u)=u\} \neq \emptyset$ if and only if $(T^n u)_n$ converges weakly to a fixed point $u_0 \in \fix T$ for any choice of $u \in \mathcal H$.
\end{proposition}

\begin{proof}(Proposition \ref{fixptitpr})
Orthogonal projections onto convex sets are strongly nonexpansive \cite[Corollary 4.2.3]{bbl}. Moreover, composition of strongly nonexpansive maps are strongly nonexpansive \cite[Lemma 2.1]{br}. By an application of Proposition \ref{bbl} we immediately have the result, since any map of the type $T(\xi) = Q(\xi) + \xi_0$ is strongly nonexpansive whenever $Q$ is (this is a simple observation from the definition of strongly nonexpansive map).
Indeed, we are looking for fixed points of $\eta = \pi_{V_2} P_{\alpha K_\psi} (\eta - (z+ u_2))$ or, equivalently, of $\xi = \underbrace{\pi_{V_2} P_{\alpha K_\psi}}_{:=Q} (\xi) - \underbrace{(z+u_2)}_{:=\xi_0}$.

\end{proof}
In Examples \ref{ex3}, we have already observed that
$$
u_1^* = \prox_{\alpha \psi(\cdot +u_2)} (z).
$$
For consistency with the terminology of generalized thresholding in signal processing, we may call the map $\prox_{\alpha \psi(\cdot +u_2)}$ an {\it oblique thresholding} and we denote it by
$$
\mathbb S_\alpha^{\psi,V_1,V_2} (z;u_2):= \prox_{\alpha \psi(\cdot +u_2)}(z).
$$
The attribute ``{\it oblique}'' emphasizes the presence of an additional subspace which acts for the computation of the thresholded solution.
By using results in \cite[Subsection 2.3]{CW} (see also \cite[II.2-3]{ET}) we can already infer that
$$
\| S_\alpha^{\psi,V_1,V_2} (z_1;u_2) -  S_\alpha^{\psi,V_1,V_2}(z_2;u_2)\|_{\mathcal{H}} \leq \| z_1 - z_2\|_{\mathcal{H}}, \quad \mbox{for all } z_1,z_2 \in V_1.
$$

\subsection{Convergence of the subspace minimization}

In light of the results of the previous subsection, the iterative algorithm (\ref{m2}) can be equivalently be rewritten as
\begin{equation}
\label{m3}
u_1^{(\ell+1)} = S_\alpha^{\psi,V_1,V_2}(u_1^{(\ell)} + \pi_{V_1} T^* (g - T u_2 - T   u_1^{(\ell)});u_2).
\end{equation}
In certain cases, e.g., in finite dimensions, the iteration can be explicitely computed by
$$
u_1^{(\ell+1)} = S_\alpha^\psi(u_1^{(\ell)} + \pi_{V_1} T^* (g - T u_2 - T   u_1^{(\ell)}) + u_2 - \eta^{(\ell)}) - u_2,
$$
where $\eta^{(\ell)} \in V_2$ is any solution of the fixed point equation
$$
\eta = \pi_{V_2} P_{\alpha K_\psi}(\eta- (u_1^{(\ell)} + \pi_{V_1} T^* (g - T u_2 - T   u_1^{(\ell)})+ u_2)).
$$
The computation of $\eta^{(\ell)}$ can be (approximatively) implemented by the algorithm (\ref{fixptit}).

\begin{theorem}\label{main2}
Assume $u_2 \in V_2^\psi$ and $\| T\| <1$. Then the iteration (\ref{m3}) converges weakly to a solution $u^*_1 \in V_1^\psi$ of (\ref{m1}) for any initial choice of $u_1^{(0)} \in V_1^\psi$.
\end{theorem}
\begin{proof}
For the sake of completeness, we report the proof of this theorem, which follows the same strategy already proposed  in the paper \cite{dadede04}, compare also similar results in \cite{CW}. In particular we want to apply Opial's fixed point theorem:
\\

\begin{theorem}[\cite{op67}]
\label{opial}
Let the mapping $A$ from $\mathcal{H}$ to $\mathcal{H}$ satisfy the following conditions:
\begin{itemize}
\item[(i)] $A$ is nonexpansive: for all $z,z' \in \mathcal H$, $\| A z - A z'\|_{\mathcal H} \leq \| z-z'\|_{\mathcal H}$;
\item[(ii)] $A$ is asymptotically regular: for all $z \in \mathcal{H}$, $\| A^{n+1} z- A^{n}z \|_{\mathcal H} \rightarrow 0$, for $n \rightarrow \infty$;
\item[(iii)] the set $\mathcal{F} =\fix A$ of fixed points of $A$ in $\mathcal{H}$ is not empty.
\end{itemize}
Then for all $z \in \mathcal{H}$, the sequence $(A^n z)_{n \in \mathbb{N}}$ converges weakly to a fixed point in $\mathcal{F}$.
\end{theorem}\\


We need to prove that $A(u_1) :=  S_\alpha^{\psi,V_1,V_2}(u_1 + \pi_{V_1} T^* (g - T u_2 - T   u_1);u_2)$ fulfills the assumptions of the Opial's theorem on $V_1$.

Step 1. As stated at the beginning of this section, there exist solutions $u^*_1 \in V_1^\psi$ to (\ref{m1}). With a similar argument to the one used to prove the equivalence of (i) and (ii) in Theorem \ref{main1}, the optimality of $u^*_1$ can be readily proved equivalent to 
$$
0 \in - \pi_{V_1} T^* (g - T u_2 - T   u_1^*)+ \eta +  \alpha \partial_{\mathcal H^\psi} \psi (u^*_1 + u_2),
$$
for some $\eta \in (V_2^\psi)'$. By adding and subtracting $u^*_1$ we obtain
$$
0 \in u^*_1 -(\underbrace{( u^*_1+ \pi_{V_1} T^* (g - T u_2 - T   u_1^*))}_{:=z}- \eta) +  \alpha \partial_{\mathcal H^\psi} \psi (u^*_1 + u_2),
$$
By applying the equivalence of (i) and (ii) in Theorem  \ref{main1} we obtain that $u_1^*$ is a fixed point of the following equation 
$$
u_1^*  = S_\alpha^{\psi,V_1,V_2}(u_1^* + \pi_{V_1} T^* (g - T u_2 - T   u_1^*);u_2),
$$
hence $\fix A \neq \emptyset$.

Step 2. The algorithm produces iterations which are asymptotically regular, i.e., $\|u_1^{(\ell+1)} - u_1^{(\ell)}\|_{\mathcal H} \to 0$. Indeed, by using  $\|T\|<1$ and $C:=1- \|T\|^2>0$, we have the following estimates
\begin{eqnarray*}
\mathcal J(u_1^{(\ell)} + u_2) &=& \mathcal J^s_1(u_1^{(\ell)} + u_2, u_1^{(\ell)})\\
&\geq& \mathcal J^s_1(u_1^{(\ell+1)} + u_2, u_1^{(\ell)})  \\
&\geq& \mathcal J^s_1(u_1^{(\ell+1)} + u_2, u_1^{(\ell+1)}) = \mathcal J(u_1^{(\ell+1)} + u_2),
\end{eqnarray*}
See also (\ref{decr}) and (\ref{decr2}) below. 
Since $(\mathcal J(u_1^{(\ell)} + u_2))_\ell$ is monotonically decreasing and bounded from below by $0$, necessarily it is a convergent sequence.
Moreover, 
$$
\mathcal J(u_1^{(\ell)} + u_2) - \mathcal J(u_1^{(\ell+1)} + u_2) \geq C \|u_1^{(\ell+1)} - u_1^{(\ell)}\|_{\mathcal H}^2,
$$
and the latter convergence implies $\|u_1^{(\ell+1)} - u_1^{(\ell)}\|_{\mathcal H} \to 0$.

Step 3. We are left with showing the nonexpansiveness of $A$. By nonexpansiveness of $ S_\alpha^{\psi,V_1,V_2}(\cdot;u_2)$ we obtain
\begin{eqnarray*}
&&\|  S_\alpha^{\psi,V_1,V_2}(u_1^1 + \pi_{V_1} T^* (g - T u_2 - T   u_1^1;u_2) -  S_\alpha^{\psi,V_1,V_2}(u_1^2 + \pi_{V_1} T^* (g - T u_2 - T   u_1^2;u_2) \|_{\mathcal H} \\
&\leq& \|u_1^1 + \pi_{V_1} T^* (g - T u_2 - T   u_1^1) - (u_1^2 + \pi_{V_1} T^* (g - T u_2 - T   u_1^2)\|_{\mathcal H} \\
&=& \| (I- \pi_{V_1} T^* T \pi_{V_1})( u_1^1 - u_1^2)\|_{\mathcal H}\\
&\leq& \|u_1^1 - u_1^2\|_{\mathcal H}
\end{eqnarray*}
In the latter inequality we used once more that $\|T\|<1$.
\end{proof}

We do not insist on conditions for the strong convergence of the iteration (\ref{m3}), which is not a relevant issue, see, e.g., \cite{CW,dateve06} for a further discussion in this direction. Indeed, the practical realization of (\ref{schw_sp})  will never solve completely the subspace minimizations. 
\\

Let us conclude this section mentioning that all the results presented here hold symmetrically for the minimization on $V_2$, and that the notations should be just adjusted accordingly.

\section{Convergence of the Sequential Alternating Subspace Minimization}

We return to the algorithm (\ref{schw_sp}). In the following we denote $u_i = \pi_{V_i} u$ for $i=1,2$. Let us explicitly express the algorithm  as follows:
 Pick an initial $V_1\oplus V_2 \ni u_1^{(0,L)}+ u_2^{(0,M)} : = u^{(0)} \in \mathcal H^\psi$, for example $u^{(0)}=0$, and iterate
\begin{equation}
\label{schw_sp:it2}
\left \{ 
\begin{array}{ll}
\left \{ 
\begin{array}{ll}
u_1^{(n+1,0)} =  u_1^{(n,L)}&\\
u_1^{(n+1,\ell+1)} =  \argmin_{u_1 \in V_1} \mathcal J_1^s(u_1+ u_2^{(n,M)}, u_1^{(n+1,\ell)}) & \ell=0,\dots, L-1\\
\end{array}\right. &\\
\left \{ 
\begin{array}{ll}
u_2^{(n+1,0)} =  u_2^{(n,M)}&\\
u_2^{(n+1,m+1)} = \argmin_{u_2 \in V_2} \mathcal J_2^s(u_1^{(n+1,L)}+ u_2, u_2^{(n+1,m)}) &m=0,\dots, M -1\\
\end{array}\right. &\\
u^{(n+1)}:=u_{1}^{(n+1,L)} + u_{2}^{(n+1,M)}.
\end{array}
\right.
\end{equation} 
Note that we do prescribe a finite number $L$ and $M$ of inner iterations for each subspace respectively.
In this section we want to prove its convergence for any choice of $L$ and $M$.\\

Observe that, for $a \in V_i$ and $\|T\|<1$, 
\begin{equation}
\| u_i- a\|_{\mathcal H}^2 - \| Tu_i- Ta\|^2_{\mathcal H} \geq C \| u_i- a\|_{\mathcal H}^2,
\end{equation}
for $C=(1-\|T\|^2)>0$. Hence
\begin{equation}
\label{decr}
\mathcal {J}(u) = \mathcal {J}_i^S(u,u_i) \leq \mathcal {J}_i^S(u,a), 
\end{equation}
and
\begin{equation}
\label{decr2}
\mathcal {J}_i^S(u,a) - \mathcal {J}_i^S(u,u_i) \geq C \|u_i-a\|^2_{\mathcal H}.
\end{equation}

\begin{theorem}[Convergence properties]
\label{weak-conv}
The algorithm in (\ref{schw_sp:it2}) produces a sequence $(u^{(n)})_{n\in \mathbb{N}}$ in  $\mathcal H^\psi$ with the following properties:
\begin{itemize}
\item[(i)] $\mathcal{J}(u^{(n)}) > \mathcal{J}(u^{(n+1)})$ for all $n \in \mathbb{N}$ (unless $u^{(n)}= u^{(n+1)}$);
\item[(ii)] $\lim_{n \to \infty} \| u^{(n+1)} -  u^{(n)}\|_{\mathcal H} =0$;
\item[(iii)] the sequence  $(u^{(n)})_{n\in \mathbb{N}}$ has subsequences which converge weakly in $\mathcal H$ and in $\mathcal H^\psi$ endowed with the topology $\tau^\psi$;
\item[(iv)] if we additionally assume, for simplicity, that $\dim \mathcal H < \infty$, $(u^{(n_k)})_{k \in \mathbb{N}}$ is a strongly converging subsequence, and  $u^{(\infty)}$ is its limit, then $u^{(\infty)}$ is a minimizer of $\mathcal J$ whenever one of the following conditions holds
\begin{itemize}
\item[(a)] $\psi(u_1^{(\infty)} + \eta_2) + \psi(u_2^{(\infty)} + \eta_1) - \psi(u_1^{(\infty)}+ u_2^{(\infty)}) \leq \psi(\eta_1 + \eta_2)$ for all $\eta_i \in V_i$, $i=1,2$;
\item[(b)] $\psi$ is differentiable at $u^{(\infty)}$ with respect to $V_i$ for one $i \in \{1,2\}$, i.e., there exists $\frac{\partial}{\partial V_i} \psi (u^{(\infty)}):=\zeta_i \in (V_i)'$ such that 
$$
\langle \zeta_i, v_i \rangle = \lim_{t \to 0} \frac{\psi(u_1^{(\infty)}+u_2^{(\infty)} + t v_i) - \psi(u_1^{(\infty)}+u_2^{(\infty)})}{t}, \mbox{ for all } v_i \in V_i.
$$
\end{itemize}  
\end{itemize}
\end{theorem}
\begin{proof}
Let us first observe that 
\begin{eqnarray*}
 \mathcal {J}(u^{(n)})= \mathcal {J}_1^S(u^{(n)}_1+ u^{(n)}_2, u^{(n)}_1)&=& \mathcal {J}^S_1(u_1^{(n,L)}  + u_2^{(n)}, u_1^{(n+1,0)}).
\end{eqnarray*}
By definition of $u_1^{(n+1,1)}$ and the minimal properties of $u_1^{(n+1,1)}$  in (\ref{schw_sp:it2}) we have
$$
\mathcal {J}_1^S(u_1^{(n,L)}  + u_2^{(n)}, u_1^{(n+1,0)}) \geq \mathcal {J}^S(u_1^{(n+1,1)}   + u_2^{(n)}, u_1^{(n+1,0)}).
$$
From (\ref{decr}) we have
$$
 \mathcal {J}^S_1(u_1^{(n+1,1)}   + u_2^{(n)}, u_1^{(n+1,0)}) \geq\mathcal {J}^S_1(u_1^{(n+1,1)}   + u_2^{(n)}, u_1^{(n+1,1)}).
$$
Putting in line these inequalities we obtain 
$$
 \mathcal {J}(u^{(n)}) \geq \mathcal {J}(u_1^{(n+1,1)}   + u_2^{(n)})
$$
In particular, from (\ref{decr2}) we have
$$
 \mathcal {J}(u^{(n)}) -     \mathcal {J}(u_1^{(n+1,1)}   + u_2^{(n)}) \geq C \| u_1^{(n+1,1)} - u_1^{(n+1,0)}\|_{\mathcal H}^2.
$$
After $L$ steps we conclude the estimate
\begin{eqnarray*}
\mathcal {J}(u^{(n)}) &\geq&\mathcal {J}(u_1^{(n+1,L)}   + u_2^{(n)}),
\end{eqnarray*}
and
$$
 \mathcal {J}(u^{(n)}) -  \mathcal {J}(u_1^{(n+1,L)}   + u_2^{(n)})\geq C \sum_{\ell=0}^{L-1}\| u_1^{(n+1,\ell+1)} - u_1^{(n+1,\ell)}\|_{\mathcal H}^2.
$$
By definition of $u_2^{(n+1,1)}$ and its minimal properties we have
\begin{eqnarray*}
&& \mathcal {J}(u_1^{(n+1,L)}   + u_2^{(n)})\geq  \mathcal {J}^S_2(u_1^{(n+1,L)}   + u_2^{(n+1,1)}, u_2^{(n+1,0)} ).
\end{eqnarray*}
By similar arguments as above we finally find the decreasing estimate
\begin{equation}
\label{decr3}
\mathcal {J}(u^{(n)}) \geq  \mathcal {J}^S_2(u_1^{(n+1,L)}   + u_2^{(n+1,M)}) = \mathcal {J}(u^{(n+1)}),
\end{equation}
and
$$
\mathcal {J}(u^{(n)}) - \mathcal {J}(u^{(n+1)}) 
$$
\begin{equation}
\label{coercive}\geq C \left ( \sum_{\ell=0}^{L-1}\| u_1^{(n+1,\ell+1)} - u_1^{(n+1,\ell)}\|_{\mathcal H}^2 + \sum_{m=0}^{M-1}\| u_2^{(n+1,m+1)} - u_2^{(n+1,m)}\|_{\mathcal H}^2 \right ).
\end{equation}
From (\ref{decr3}) we have $\mathcal {J}(u^{(0)}) \geq \mathcal {J}(u^{(n)})$. By the coerciveness condition (C) $(u^{(n)})_{n \in \mathbb{N}}$ is uniformly bounded in $\mathcal H^\psi$, hence there exists a $\mathcal H$-weakly- and $\tau^\psi$-convergent subsequence $(u^{(n_j)})_{j \in \mathbb{N}}$. 
Let us denote $u^{(\infty)}$ the weak limit of the subsequence. For simplicity, we rename such a subsequence by $(u^{(n)})_{n \in \mathbb{N}}$. Moreover, since the sequence $(\mathcal {J}(u^{(n)}))_{n \in \mathbb{N}}$ is monotonically decreasing and bounded from below by 0, it is also convergent. From (\ref{coercive}) and the latter convergence we deduce 
\begin{equation}
\label{asymp_reg}
\left ( \sum_{\ell=0}^{L-1}\| u_1^{(n+1,\ell+1)} - u_1^{(n+1,\ell)}\|_{\mathcal H}^2 + \sum_{m=0}^{M-1}\| u_2^{(n+1,m+1)} - u_2^{(n+1,m)}\|_{\mathcal H}^2 \right ) \rightarrow 0, \quad n \to \infty.
\end{equation}
In particular, by the standard inequality $(a^2+b^2) \geq \frac{1}{2} (a+b)^2$ for $a,b>0$ and the triangle inequality, we have also
\begin{equation}
\label{asymp_reg2}
\| u^{(n)} -  u^{(n+1)} \|_{\mathcal H} \rightarrow 0, \quad n\to \infty.
\end{equation}
We would like now to show that the following outer lower semicontinuity holds
$$
0 \in \lim_{n \to \infty} \partial \mathcal {J}(u^{(n)}) \subset \partial \mathcal {J}(u^{(\infty)}).
$$
For this we need to assume that $\mathcal H$-weakly- and $\tau^\psi-$convergences do imply strong convergence in $\mathcal H$. This is the case, e.g., when $\dim(\mathcal H) < \infty$. The optimality condition for $  u_1^{(n+1,L)}$ is equivalent to 
\begin{equation}
\label{dinc1}
0 \in u_1^{(n+1,L)} - z_1^{(n+1)} + \alpha \partial_{V_1} \psi(\cdot + u_2^{(n,M)})(u_1^{(n+1,L)}),
\end{equation}
where
$$
z_1^{(n+1)}:= u_1^{(n+1,L-1)} + \pi_{V_1} T^* ( g - T  u_2^{(n,M)} - T u_1^{(n+1,L-1)}).
$$
Analogously we have
\begin{equation}
\label{dinc2}
0 \in u_2^{(n+1,M)} - z_2^{(n+1)} + \alpha \partial_{V_2} \psi(\cdot + u_1^{(n+1,L)})(u_2^{(n+1,M)}),
\end{equation}
where
$$
z_2^{(n+1)}:= u_2^{(n+1,M-1)} + \pi_{V_2} T^* ( g - T  u_1^{(n+1,L)} - T u_2^{(n+1,M-1)}).
$$
Due to the strong convergence of the sequence $u^{(n)}$ and by (\ref{asymp_reg}) we have the following limits for $n \to \infty$
$$
\xi_1^{(n+1)}:= u_1^{(n+1,L)} - z_1^{(n+1)} \to \xi_1:=-\pi_{V_1}  T^* ( g - T  u_2^{(\infty)} - T u_1^{(\infty)}) \in V_1,
$$
$$
\xi_2^{(n+1)}:= u_2^{(n+1,M)} - z_2^{(n+1)} \to \xi_2:=-\pi_{V_2}  T^* ( g - T  u_2^{(\infty)} - T u_1^{(\infty)}) \in V_2,
$$
and
$$
\xi_1^{(n+1)}+ \xi_2^{(n+1)} \to \xi:=T^* (  T u^{(\infty)}- g ).
$$
Moreover, we have
$$
- \frac{1}{\alpha} \xi_1^{(n+1)} \in \partial_{V_1} \psi(\cdot + u_2^{(n,M)})(u_1^{(n+1,L)}),
$$
meaning that
$$
\langle - \frac{1}{\alpha} \xi_1^{(n+1)}, \eta_1 - u_1^{(n+1,L)} \rangle + \psi(u_1^{(n+1,L)}+ u_2^{(n,M)}) \leq \psi(\eta_1+ u_2^{(n,M)}), \quad \mbox{for all }\eta_1 \in V_1.
$$
Analogously we have

$$
\langle - \frac{1}{\alpha} \xi_2^{(n+1)}, \eta_2 - u_2^{(n+1,M)} \rangle + \psi(u_1^{(n+1,L)}+ u_2^{(n+1,M)}) \leq \psi(\eta_2+ u_1^{(n+1,L)}), \quad \mbox{for all }\eta_2 \in V_2.
$$
By taking the limits for $n \to \infty$  and by (\ref{asymp_reg}) we obtain 
\begin{equation}
\label{dincl1}
\langle - \frac{1}{\alpha} \xi_1, \eta_1 - u_1^{(\infty)} \rangle + \psi(u^{(\infty)}) \leq \psi(\eta_1+ u_2^{(\infty)}), \quad \mbox{for all }\eta_1 \in V_1.
\end{equation}
\begin{equation}
\label{dincl2}
\langle - \frac{1}{\alpha} \xi_2, \eta_2 - u_2^{(\infty)} \rangle + \psi(u_1^{(\infty)}) \leq \psi(\eta_2+ u_1^{(\infty)}), \quad \mbox{for all }\eta_2 \in V_2.
\end{equation}
These latter conditions are rewritten in vector form as
\begin{equation}
\label{vec}
0 \in \left ( \begin{array}{l}
\xi_1\\
\xi_2
\end{array}
\right ) + \alpha
\left ( \partial_{V_1} \psi(\cdot + u_2^{(\infty)})(u_1^{(\infty)}) \times
\partial_{V_2} \psi(\cdot + u_1^{(\infty)})(u_2^{(\infty)}) \right).
\end{equation}
Observe now that 
$$
2 \xi + 2 \alpha \partial_{\mathcal H} \psi(u^{(\infty)}) = 2 T^* (  T u^{(\infty)}- g )+  2 \alpha \partial_{\mathcal H} \psi(u^{(\infty)}) = \partial \mathcal J(u^{(\infty)}).
$$
If $0 \in \xi +  \alpha \partial_{\mathcal H} \psi(u^{(\infty)})$ then we would have the wanted minimality condition.
While the inclusion
$$
\partial_{\mathcal H} \psi(u^{(\infty)}) \subset \partial_{V_1} \psi(\cdot + u_2^{(\infty)})(u_1^{(\infty}) \times
\partial_{V_2} \psi(\cdot + u_1^{(\infty)})(u_2^{(\infty)}),
$$
easily follows from the definition of a subdifferential, the converse inclusion, which would imply from (\ref{vec}) the wished minimality condition, does not hold in general.
Thus, we show the converse inclusion under one of the following two conditions:
\begin{itemize}
\item[(a)] $\psi(u_1^{(\infty)} + \eta_2) + \psi(u_2^{(\infty)} + \eta_1) - \psi(u_1^{(\infty)}+ u_2^{(\infty)}) \leq \psi(\eta_1 + \eta_2)$ for all $\eta_i \in V_i$, $i=1,2$;
\item[(b)] $\psi$ is differentiable at $u^{(\infty)}$ with respect to $V_i$ for one $i \in \{1,2\}$, i.e., there exists $\frac{\partial}{\partial V_i} \psi (u^{(\infty)}):=\zeta_i \in (V_i)'$ such that 
$$
\langle \zeta_i, v_i \rangle = \lim_{t \to 0} \frac{\psi(u_1^{(\infty)}+u_2^{(\infty)} + t v_i) - \psi(u_1^{(\infty)}+u_2^{(\infty)})}{t}, \mbox{ for all } v_i \in V_i.
$$
\end{itemize}
Let us start with condition (a).
We want to show that 
$$
\langle- \frac{1}{\alpha}  \xi, \eta - u^{(\infty)} \rangle + \psi(u^{(\infty)}) \leq \psi(\eta), \quad \mbox{for all } \eta \in \mathcal H,
$$
or, equivalently, that
$$
\langle - \frac{1}{\alpha} \xi_1, \eta_1 - u_1^{(\infty)} \rangle + \langle - \frac{1}{\alpha} \xi_2, \eta_2 - u^{(\infty)}_2 \rangle + \psi(u^{(\infty)}_1+ u^{(\infty)}_2) \leq \psi(\eta_1+\eta_2), \quad \mbox{for all } \eta_i \in V_i,
$$
By the differential inclusions (\ref{dincl1}) and (\ref{dincl2}) we have
$$
\langle - \frac{1}{\alpha} \xi_1, \eta_1 - u_1^{(\infty)} \rangle + \langle - \frac{1}{\alpha} \xi_2, \eta_2 - u^{(\infty)}_2 \rangle + 2 \psi(u^{(\infty)}_1+ u^{(\infty)}_2) \leq \psi(u_1^{(\infty)} + \eta_2) + \psi(u_2^{(\infty)} + \eta_1), \quad \mbox{for all } \eta_i \in V_i,
$$
hence
\begin{eqnarray*}
&&\langle - \frac{1}{\alpha} \xi_1, \eta_1 - u_1^{(\infty)} \rangle + \langle - \frac{1}{\alpha} \xi_2, \eta_2 - u^{(\infty)}_2 \rangle + \psi(u^{(\infty)}_1+ u^{(\infty)}_2)\\
 &\leq& \psi(u_1^{(\infty)} + \eta_2) + \psi(u_2^{(\infty)} + \eta_1)-  \psi(u^{(\infty)}_1+ u^{(\infty)}_2), \quad \mbox{for all } \eta_i \in V_i.
\end{eqnarray*}
An application of condition (a) concludes the proof of the wanted differential inclusion.\\

Let us show the inclusion now under the assumption of condition (b). Without loss of generality, we assume that $\psi$ is differentiable at $u^{(\infty)}$ with respect to $V_2$.
First of all we define $\tilde \psi(u_1,u_2) := \psi(u_1+u_2)$.
Since $\psi$ is convex, by an application of \cite[Corollary 10.11]{rowe98}, we have
$$
\partial_{V_1} \psi(\cdot + u_2)(u_1) \simeq \partial_{u_1} \tilde \psi(u_1,u_2) = \{ \zeta_1 \in V_1' : \exists \zeta_2 \in V_2': (\zeta_1,\zeta_2)^T \in \partial \tilde \psi(u_1,u_2) \simeq \partial_{\mathcal H} \psi(u_1+u_2)\}.
$$
Since $\psi$ is differentiable at $u^{(\infty)}$ with respect to $V_2$, for any $ (\zeta_1,\zeta_2)^T \in \partial \tilde \psi(u_1,u_2) \simeq \partial_{\mathcal H} \psi(u_1+u_2)$ we have necessarily
$\zeta_2 = \frac{\partial}{\partial V_2} \psi (u^{(\infty)})$ as the unique member of $\partial_{V_2} \psi(\cdot + u_1^{(\infty)})(u_2^{(\infty)})$. Hence, the following inclusion must also hold
\begin{eqnarray*}
0 &\in& \left ( \begin{array}{l}
\xi_1\\
\xi_2
\end{array}
\right ) + \alpha
\left ( \partial_{V_1} \psi(\cdot + u_2^{(\infty)})(u_1^{(\infty}) \times
\partial_{V_2} \psi(\cdot + u_1^{(\infty)})(u_2^{(\infty)}) \right) \\
&\subset& \left ( \begin{array}{l}
\xi_1\\
\xi_2
\end{array}
\right ) + \alpha \partial_{V_1 \times V_2} \tilde \psi(u_1,u_2) \\
&\simeq& \xi + \alpha  \partial_{\mathcal H} \psi(u^{(\infty)}).
\end{eqnarray*}
\end{proof}
\begin{remark}
Observe that, by choosing $\eta_1=\eta_2=0$, condition (a) and $(\Psi1)$ imply that
$$
\psi (u_1^{(\infty)}) + \psi (u_2^{(\infty)}) \leq \psi (u_1^{(\infty)}+u_2^{(\infty)})
$$ 
The sublinearity $(\Psi2)$ finally implies the splitting
$$
\psi (u_1^{(\infty)}) + \psi (u_2^{(\infty)}) = \psi (u_1^{(\infty)}+u_2^{(\infty)})
$$
Conversely, if $\psi (v_1) + \psi (v_2) = \psi (v_1+ v_2)$ for all $v_i \in V_i$, $i=1,2$, then condition (a) easily follows. As previously discussed, this latter splitting condition  holds only in special cases.
Also condition (b) is not in practice always verified, as we will illustrate with numerical examples in Section \ref{accel}. Hence, we can affirm that in general we cannot expect convergence of the algorithm to minimizers of $\mathcal J$, although it certainly converges to points for which  $\mathcal J$ is smaller than the starting choice $\mathcal J(u^{(0)})$. However, as we will show in the numerical experiments related to total variation minimization (Section \ref{domdec}), the computed limit can be very close to the expected minimizer.   
\end{remark}
\section{A Parallel Alternating Subspace Minimization and its Convergence}

The most immediate modification to (\ref{schw_sp:it2}) is provided by substituting $u_1^{(n,L)}$ instead of $u_1^{(n+1,L)}$ in the second iteration, producing the following parallel algorithm:
\begin{equation}
\label{schw_sp:it3}
\left \{ 
\begin{array}{ll}
\left \{ 
\begin{array}{ll}
u_1^{(n+1,0)} =  u_1^{(n,L)}&\\
u_1^{(n+1,\ell+1)} =  \argmin_{u_1 \in V_1} \mathcal J_1^s(u_1+ u_2^{(n,M)}, u_1^{(n+1,\ell)}) & \ell=0,\dots, L-1\\
\end{array}\right. &\\
\left \{ 
\begin{array}{ll}
u_2^{(n+1,0)} =  u_2^{(n,M)}&\\
u_2^{(n+1,m+1)} = \argmin_{u_2 \in V_2} \mathcal J_2^s(u_1^{(n,L)}+ u_2, u_2^{(n+1,m)}) &m=0,\dots, M -1\\
\end{array}\right. &\\
u^{(n+1)}:=u_{1}^{(n+1,L)} + u_{2}^{(n+1,M)}.
\end{array}
\right.
\end{equation} 
Unfortunately, this modification violates the monotonicity property $\mathcal{J}(u^{(n)}) \geq \mathcal{J}(u^{(n+1)})$ and the overall algorithm does not converge in general.
In order to preserve the monotonicity of the iteration with respect to $\mathcal{J}$ a simple trick can be applied, i.e., modifying $u^{(n+1)}:=u_{1}^{(n+1,L)} + u_{2}^{(n+1,M)}$ by the average of the current iteration and the previous one. This leads to the following parallel algorithm:

\begin{equation}
\label{schw_sp:it4}
\left \{ 
\begin{array}{ll}
\left \{ 
\begin{array}{ll}
u_1^{(n+1,0)} =  u_1^{(n,L)}&\\
u_1^{(n+1,\ell+1)} =  \argmin_{u_1 \in V_1} \mathcal J_1^s(u_1+ u_2^{(n,M)}, u_1^{(n+1,\ell)}) & \ell=0,\dots, L-1\\
\end{array}\right. &\\
\left \{ 
\begin{array}{ll}
u_2^{(n+1,0)} =  u_2^{(n,M)}&\\
u_2^{(n+1,m+1)} = \argmin_{u_2 \in V_2} \mathcal J_2^s(u_1^{(n,L)}+ u_2, u_2^{(n+1,m)}) &m=0,\dots, M -1\\
\end{array}\right. &\\
u^{(n+1)}:=\frac{u_{1}^{(n+1,L)} + u_{2}^{(n+1,M)}+u^{(n)}}{2}.
\end{array}
\right.
\end{equation} 
In this section we prove similar convergence properties of this algorithm as for (\ref{schw_sp:it2}).
\begin{theorem}[Convergence properties]
\label{weak-conv2}
The algorithm in (\ref{schw_sp:it4}) produces a sequence $(u^{(n)})_{n\in \mathbb{N}}$ in  $\mathcal H^\psi$ with the following properties:
\begin{itemize}
\item[(i)] $\mathcal{J}(u^{(n)}) > \mathcal{J}(u^{(n+1)})$ for all $n \in \mathbb{N}$ (unless $u^{(n)}= u^{(n+1)}$);
\item[(ii)] $\lim_{n \to \infty} \| u^{(n+1)} -  u^{(n)}\|_{\mathcal H} =0$;
\item[(iii)] the sequence  $(u^{(n)})_{n\in \mathbb{N}}$ has subsequences which converge weakly in $\mathcal H$ and in $\mathcal H^\psi$ endowed with the topology $\tau^\psi$;
\item[(iv)] if we additionally assume that $\dim \mathcal H < \infty$, $(u^{(n_k)})_{k \in \mathbb{N}}$ is a strongly converging subsequence, and  $u^{(\infty)}$ is its limit, then $u^{(\infty)}$ is a minimizer of $\mathcal J$ whenever one of the following conditions holds
\begin{itemize}
\item[(a)] $\psi(u_1^{(\infty)} + \eta_2) + \psi(u_2^{(\infty)} + \eta_1) - \psi(u_1^{(\infty)}+ u_2^{(\infty)}) \leq \psi(\eta_1 + \eta_2)$ for all $\eta_i \in V_i$, $i=1,2$;
\item[(b)] $\psi$ is differentiable at $u^{(\infty)}$ with respect to $V_i$ for one $i \in \{1,2\}$, i.e., there exists $\frac{\partial}{\partial V_i} \psi (u^{(\infty)}):=\zeta_i \in (V_i)'$ such that 
$$
\langle \zeta_i, v_i \rangle = \lim_{t \to 0} \frac{\psi(u_1^{(\infty)}+u_2^{(\infty)} + t v_i) - \psi(u_1^{(\infty)}+u_2^{(\infty)})}{t}, \mbox{ for all } v_i \in V_i.
$$
\end{itemize}  
\end{itemize}
\end{theorem}
\begin{proof}
With the same argument as in the proof of Theorem \ref{weak-conv}, we obtain
$$
 \mathcal {J}(u^{(n)}) -  \mathcal {J}(u_1^{(n+1,L)}   + u_2^{(n)})\geq C \sum_{\ell=0}^{L-1}\| u_1^{(n+1,\ell+1)} - u_1^{(n+1,\ell)}\|_{\mathcal H}^2.
$$
and
$$
 \mathcal {J}(u^{(n)}) -  \mathcal {J}(u_1^{(n)}   + u_2^{(n+1,M)})\geq C  \sum_{m=0}^{M-1}\| u_2^{(n+1,m+1)} - u_2^{(n+1,m)}\|_{\mathcal H}^2.
$$
Hence, by summing and halving  
\begin{eqnarray*}
&& \mathcal {J}(u^{(n)}) -  \frac{1}{2}( \mathcal {J}(u_1^{(n+1,L)}   + u_2^{(n)}) + \mathcal {J}(u_1^{(n)}   + u_2^{(n+1,M)})) \\
&\geq& \frac{C}{2} \left ( \sum_{\ell=0}^{L-1}\| u_1^{(n+1,\ell+1)} - u_1^{(n+1,\ell)}\|_{\mathcal H}^2 + \sum_{m=0}^{M-1}\| u_2^{(n+1,m+1)} - u_2^{(n+1,m)}\|_{\mathcal H}^2 \right ).
\end{eqnarray*}
By convexity we have
\begin{eqnarray*}
\left \| T u^{(n+1)}- g \right\|^2_{\mathcal H} &=& \left \| T \left (\frac{(u_1^{(n+1,L)} + u_2^{(n+1,M)}) + u^{(n)}}{2} \right ) - g \right\|^2_{\mathcal H} \\
&\leq & \frac{1}{2} \| T(u^{(n+1,L)}_1 +  u^{(n)}_2 ) -g \|^2_{\mathcal H} + \frac{1}{2} \| T(u^{(n)}_1 + u^{(n+1,M)}_2 ) -g \|^2_{\mathcal H}.
\end{eqnarray*}
Moreover, by sublinearity ($\Psi$2) and 1-homogeneity  ($\Psi$3) we have
\begin{eqnarray*}
\psi(u^{(n+1)}) &\leq& \frac{1}{2} \left ( \psi(u_1^{(n+1,L)} + u_2^{(n)}) + \psi(  u^{(n)}_1 + u_2^{(n+1,M)}) \right )
\end{eqnarray*}
By the last two inequalities we immediately show that
$$
\mathcal  J (u^{(n+1)}) \leq \frac{1}{2} \left (\mathcal {J}(u_1^{(n+1,L)}   + u_2^{(n)})+ \mathcal {J}(u^{(n)}_1   + u_2^{(n+1,M)}) \right ),
$$
hence
$$
\mathcal {J}(u^{(n)})- \mathcal  J (u^{(n+1)})$$\begin{equation}
\label{coerc3}\geq \frac{C}{2} \left( \sum_{\ell=0}^{L-1}\| u_1^{(n+1,\ell+1)} - u_1^{(n+1,\ell)}\|_{\mathcal H}^2 +\sum_{\ell=0}^{M-1}\| u_2^{(n+1,\ell+1)} - u_2^{(n+1,\ell)}\|_{\mathcal H}^2 \right) \geq 0.
\end{equation}
Since the sequence $(\mathcal {J}(u^{(n)}))_{n \in \mathbb{N}}$ is monotonically decreasing and bounded from below by 0, it is also convergent. From (\ref{coerc3}) and the latter convergence we deduce 
\begin{equation}
\label{asymp_reg3}
\left ( \sum_{\ell=0}^{L-1}\| u_1^{(n+1,\ell+1)} - u_1^{(n+1,\ell)}\|_{\mathcal H}^2 + \sum_{m=0}^{M-1}\| u_2^{(n+1,m+1)} - u_2^{(n+1,m)}\|_{\mathcal H}^2 \right ) \rightarrow 0, \quad n \to \infty.
\end{equation}
In particular, by the standard inequality $(a^2+b^2) \geq \frac{1}{2} (a+b)^2$ for $a,b>0$ and the triangle inequality, we have also
\begin{eqnarray*}
\sum_{\ell=0}^{L-1}\| u_1^{(n+1,\ell+1)} - u_1^{(n+1,\ell)}\|_{\mathcal H}^2 &\geq& C'' \left ( \sum_{\ell=0}^{L-1}\| u_1^{(n+1,\ell+1)} - u_1^{(n+1,\ell)}\|_{\mathcal H} \right )^2\\
&\geq & C''\| u^{(n+1,L)}_1 - u^{(n)}_1\|^2_{\mathcal H} \\
&=&  C''\|u^{(n+1,L)}_1 +u^{(n)}_1 - 2 u^{(n)}_1\|^2_{\mathcal H}.
\end{eqnarray*}
Analogously we have
\begin{eqnarray*}
\sum_{\ell=0}^{M-1}\| u_2^{(n+1,\ell+1)} - u_2^{(n+1,\ell)}\|_{\mathcal H}^2 &\geq& C''\| u_2^{(n+1,M)} +u_2^{(n)} - 2 u_2^{(n)}\|^2_{\mathcal H}.
\end{eqnarray*}
By denoting $C''=\frac{1}{2} C'''$ we obtain
\begin{eqnarray*}
&&\frac{C}{2} \left( \sum_{\ell=0}^{L-1}\| u_1^{(n+1,\ell+1)} - u_1^{(n+1,\ell)}\|_{\mathcal H}^2 +\sum_{\ell=0}^{M-1}\| u_2^{(n+1,\ell+1)} - u_2^{(n+1,\ell)}\|_{\mathcal H}^2 \right) \\&\geq & \frac{C C'''}{4} \|  u_1^{(n+1,L)} +  u_2^{(n+1,M)} + u^{(n)} - 2  u^{(n)}\|^2_{\mathcal H} \\
&\geq & C C''' \| u^{(n+1)} - u^{(n)}\|^2_{\mathcal H}. 
\end{eqnarray*}
Therefore, we finally have
\begin{equation}
\label{asymp_reg4}
\| u^{(n)} -  u^{(n+1)} \|_{\mathcal H} \rightarrow 0, \quad n\to \infty.
\end{equation}
The rest of the proof follows analogous arguments as in that of Theorem \ref{weak-conv}.
\end{proof}

\section{Applications and Numerics}

In this section we present two nontrivial applications of the theory and algorithms illustrated in the previous sections to Examples \ref{ex1}.

\subsection{Domain decomposition methods for total variation minimization}
\label{domdec}
In the following we consider the minimization of the functional $\mathcal J$ in the setting of Examples \ref{ex1}.1. Namely, let $\Omega\subset\mathbb{R}^d$, for $d=1,2$, be a bounded open set with Lipschitz boundary. We are interested in the case when $\mathcal{H}=L^2(\Omega)$, $\mathcal{H}^\psi=BV(\Omega)$ and $\psi(u)=V(u,\Omega)$. Then the domain decomposition $\Omega=\Omega_1\cup\Omega_2$ as described in Examples \ref{ex2}.1 induces the space splitting into $V_i :=\{ u \in L^2(\Omega) : \textrm{supp}(u) \subset \Omega_i \},$ and $V_i^\psi=BV(\Omega)\cap V_i, \quad i=1,2$. In particular, we can consider multiple subspaces, since the algorithms and their analysis presented in the previous sections can be easily generalized to these cases, see \cite[Section 6]{fo07}. As before $u_{\Omega_i}=\pi_{V_i}(u) = 1_{\Omega_i} u$ is the orthogonal projection onto $V_i$.
\\
To exemplify the kind of difficulties one may encounter in the numerical treatment of the interfaces $\partial \Omega_i \cap \partial \Omega_j$, we present first an approach based on the direct discretization of the subdifferential of $\mathcal J$ in this setting.
We show that this method can work properly in many cases, but it fails in others, even in simple 1D examples, due to the raising of exceptions which cannot be captured by this formulation. Instead of insisting on dealing with these exceptions and strengthening the formulation, we show then that the general theory and algorithms  previously presented work properly and deal well with interfaces both for $d=1,2$.


\subsubsection{The ``naive'' direct approach}

In light of (\ref{amb}), the first subiteration in \eqref{schw_sp} is given by
$$
u_1^{(n+1)} \approx \textrm{argmin}_{v_1 \in V_1}  \mathcal  \| T (v_1 +u_2^{(n)}) - g \|_{ L^2(\Omega)}^2 + 2 \alpha \left (|D(v_1)|(\Omega_1)+ \int_{\partial\Omega_1\cap\partial\Omega_2} \left|v_1^+-u_2^{(n)-}\right|\; d \mathcal H_{d-1} \right ).
$$
We would like to dispose of conditions to characterize subdifferentials of functionals of the type
\begin{eqnarray*}
{\Gamma}(u) =  \left|D(u)\right|(\Omega) + \underbrace{\int_\theta \left|u^{+}-z\right|\; d \mathcal H_{d-1}}_{\mbox{interface condition}},
\end{eqnarray*}
where $\theta\subset\partial\Omega$, in order to handle the boundary conditions that are imposed at the interface.
Since we are interested in emphasizing the difficulties of this approach, we do not insist on the details of the rigorous derivation of these conditions, and we limit ourself to mention the main facts.\\
It is well known \cite[Proposition 4.1]{Ve01} that, if no interface condition is present, $\xi \in \partial | D (\cdot) |(\Omega) (u)$ implies
\[\left\{\begin{array}{l l}
\xi = -\nabla\cdot(\frac{\nabla u}{|\nabla u|}) & \textrm{ in }\Omega\\
\frac{\nabla u}{|\nabla u|} \cdot \nu = 0 & \textrm{ on }\partial\Omega.\\
\end{array}\right.\]
The previous conditions do not fully characterize $\xi \in \partial | D (\cdot) |(\Omega) (u)$, additional conditions would be required \cite{AK02, Ve01}, but the latter are, unfortunately, hardly numerically implementable. This lacking approach is the source of the failures of this direct method.  The presence of the interface further modifies and deteriorates this situation and for $\partial \Gamma (u) \neq \emptyset$ we need to enforce
\begin{eqnarray*}
\int_{\partial\Omega} \frac{\nabla u}{|\nabla u|} \cdot \nu w\; d\mathcal H_{d-1} + \lim_{s\rightarrow 0}\left (\int_\theta \frac{|(u+ws)^+-z|-|u^+-z|}{s} \; d\mathcal H_{d-1} \right ) \geq 0, \;\;\forall w\in C^\infty(\bar{\Omega}).
\end{eqnarray*}
The latter condition is implied by the following natural boundary conditions:

\begin{equation}
\label{inclinter}
\left\{\begin{array}{l l}
\frac{\nabla u}{|\nabla u|} \cdot \nu = 0 & \textrm{ on }\partial\Omega\setminus\theta,\\
-\frac{\nabla u}{|\nabla u|} \cdot \nu\in \partial | \cdot | (u^+-z) & \textrm{ on }\theta.
\end{array}\right.
\end{equation}

Note that the conditions above are again not sufficient to characterize elements in the subdifferential of $\Gamma$.

\subsubsection{Implementation of the subdifferential approach in $TV-L^2$ interpolation}
Let  $D\subset\Omega \subset\mathbb{R}^d$ open and bounded 
domains with Lipschitz boundaries. We assume that a function $g \in L^2(\Omega)$ is given only on $\Omega \setminus D$, possibly with noise disturbance. The problem is to reconstruct a function $u$ in the damaged domain $D\subset\Omega$ which nearly coincides with $g$ on $\Omega \setminus D$. In 1D this is a classical interpolation problem, in 2D has taken the name of ``inpainting'' due to its applications in image restoration.
 $TV$-interpolation/inpainting with $L^2$ fidelity is solved by minimization of the functional
\begin{eqnarray}
\label{tvinp}
J(u) = \left\|1_{\Omega\setminus D}(u-g)\right\|^2_{L^2(\Omega)} + 2 \alpha |D(u)|(\Omega),
\end{eqnarray}
where $1_{\Omega\setminus D}$ denotes the characteristic function of $\Omega\setminus D$. Hence, in this case $T$ is the multiplier operator $T u = 1_{\Omega\setminus D} u$.
We consider in the following the problem for $d=1$ so that $\Omega=(a,b)$ is an interval. We may want to minimize  (\ref{tvinp}) iteratively by a subgradient descent method,
\begin{eqnarray}\label{globprob}
\begin{array}{l l}
\frac{u^{(n+1)}-u^{(n)}}{\tau} = -\nabla\cdot (\frac{\nabla u^{(n+1)}}{|\nabla u^{(n)}|}) + 2\lambda(u^{(n)}-g) & \textrm{ in }\Omega\\
\frac{1}{|\nabla u^{(n)}|} \frac{\partial u^{(n+1)}}{\partial n} = 0 & \textrm{ on }\partial\Omega,
\end{array}
\end{eqnarray}
where
\begin{eqnarray*}
\lambda(x) = \begin{cases}
\lambda_0  = \frac{1}{4 \alpha}& \Omega\setminus{D} \\
0 & D.
\end{cases}
\end{eqnarray*}
We can also attempt the minimization by the following domain decomposition algorithm: We split $\Omega$ into two intervals $\Omega=\Omega_1\cup\Omega_2$ and define two alternating minimizations on $\Omega_1$ and $\Omega_2$ with interface $\theta=\partial \Omega_1\cap \partial \Omega_2$
\[\begin{array}{l l}
\frac{u_1^{(n+1)}-u_1^{(n)}}{\tau} = -\nabla\cdot (\frac{\nabla u_1^{(n+1)}}{|\nabla u_1^{(n)}|}) + 2\lambda_1(u_1^{(n)}-g) & \textrm{ in }\Omega_1\\
\frac{1}{|\nabla u_1^{(n)}|} \frac{\partial u_1^{(n+1)}}{\partial n} = 0 & \textrm{ on }\partial\Omega_1\setminus\theta\\
-\frac{1}{|\nabla u_1^{(n)}|} \frac{\partial u_1^{(n+1)}}{\partial n} \in \partial \left |\cdot \right| (u_1^{(n+1)+}-u_2^{(n)-}) & \textrm{ on }\theta,
\end{array}\]
and
\[\begin{array}{l l}
\frac{u_2^{(n+1)}-u_2^{(n)}}{\tau} = -\nabla\cdot (\frac{\nabla u_2^{(n+1)}}{|\nabla u_2^{(n)}|}) + 2\lambda_2(u_2^{(n)}-g) & \textrm{ in }\Omega_2\\
\frac{1}{|\nabla u_2^{(n)}|} \frac{\partial u_2^{(n+1)}}{\partial n} = 0 & \textrm{ on }\partial\Omega_2\setminus\theta\\
\frac{1}{|\nabla u_2^{(n)}|} \frac{\partial u_2^{(n+1)}}{\partial n} \in \partial \left |\cdot \right | (u_2^{(n+1)+}-u_1^{(n+1)-}) & \textrm{ on }\theta.
\end{array}\]
In this setting $u_i$ denotes the restriction of $u\in BV(\Omega)$ to $\Omega_i$. The fitting parameter $\lambda$ is also spÃlit accordingly into $\lambda_1$ and $\lambda_2$ on $\Omega_1$ and $\Omega_2$ respectively. Note that we enforced the interface conditions (\ref{inclinter}), with the hope to match correctly the solution at the internal boundaries.
\\

The discretization in space is done by finite differences. We only explain the details for the first subproblem on $\Omega_1$ because the procedure is analogous for the second one. Let $i=1,\dots,N$ denote the space nodes supported in  $\Omega_1$. We denote $h= \frac{b-a}{N}$ and $u(i) := u(i\cdot h)$. The gradient and the divergence operator are discretized by backward differences and forward differences respectively,
\begin{eqnarray*}
& & \nabla u(i) = \frac{1}{h}(u(i)-u(i-1))\\
& & \nabla \cdot u(i) = \frac{1}{h}(u(i+1)-u(i)) \\
& & |\nabla u|(i) = \sqrt{\epsilon^2 + \frac{1}{h^2}(u(i)-u(i-1))^2},
\end{eqnarray*}
for $i=2,\ldots, N-1$. The discretized equation on $\Omega_1$ turns out to be
\begin{eqnarray*}
u_1^{(n+1)}(i) &=& u_1^{(n)}(i) + 2\tau\lambda(i)(u_1^{(n)}(i)-g(i)) + \frac{\tau}{h^2}\left (\frac{u_1^{(n+1)}(i+1)-u_1^{(n+1)}(i)}{c_1^n(i+1)} \right .\\
& & \left .- \frac{u_1^{(n+1)}(i)-u_1^{(n+1)}(i-1)}{c_1^n(i)} \right ),
\end{eqnarray*}
with $c_1^n(i) = \sqrt{\epsilon^2+(u_1^{(n)}(i)-u_1^{(n)}(i-1))^2/h^2}$ and $i=2,\ldots,N-1$. The Neumann boundary conditions on the external portion of the boundary are enforced by 
\begin{eqnarray*}\frac{1}{c_1^n(1)}u_1(1)=\frac{1}{c_2^n(2)}u_1(2).
\end{eqnarray*}
The interface conditions on the internal boundaries are computed by solving the following subdifferential inclusion
\begin{eqnarray*}
-(u_1^{(n+1)}(N)-u_1^{(n+1)}(N-1))\in c_1^n(N)\; h \cdot\partial |\cdot | (u_2^{(n)}(N)-u_1^{(n+1)}(N)).
\end{eqnarray*}
For the solution of this subdifferential inclusion we recall that the soft-thresholded $u=S_\alpha(x)$ \eqref{softthr} provides the unique solution of the subdifferential inclusion $0\in (u-x) + \alpha \partial |\cdot|(u)$. We reformulate our subdifferential inclusion as 
\begin{eqnarray*}
0\in \left[v-(u_2^{(n)}(N)-u_1^{(n+1)}(N-1))\right] +  c_1^n(N) h \cdot\partial |\cdot| (v),
\end{eqnarray*}
with $v:=u_2^{(n)}(N)-u_1^{(n+1)}(N)$ and get
\begin{eqnarray*}
v=S_{c_1^n(N) h}(u_2^{(n)}(N)-u_1^{(n+1)}(N-1)).
\end{eqnarray*}
Therefore the interface condition on $\theta$ reads as $u_1^{(n+1)}(N) = u_2^{(n)}(N)-v$.
\\

In the left column of Figure \ref{fignew} three one dimensional signals are considered. The right column shows the result of the application of the domain decomposition method for total variation minimization described above. The support of the signals is split in two intervals. The interface developed by the two intervals is marked by a red dot. In all three examples we fixed $\lambda_0=1$ and $\tau=1/2$. The first example \ref{l1}-\ref{l2} shows a step function which has its step directly at the interface of the two intervals. The total variation minimization \eqref{globprob} is applied with $D=\emptyset$. This example confirms that jumps are preserved at the interface of the two domains. The second and third example \ref{l3}-\ref{l6} present the behaviour of the algorithm when interpolation across the interface is performed, i.e., $D \neq \emptyset$. In the example \ref{l3}-\ref{l4} the computation at the interface is correctly performed. But the computation at the interface clearly fails in the last example \ref{l5}-\ref{l6}, compare the following remark.

\begin{figure}
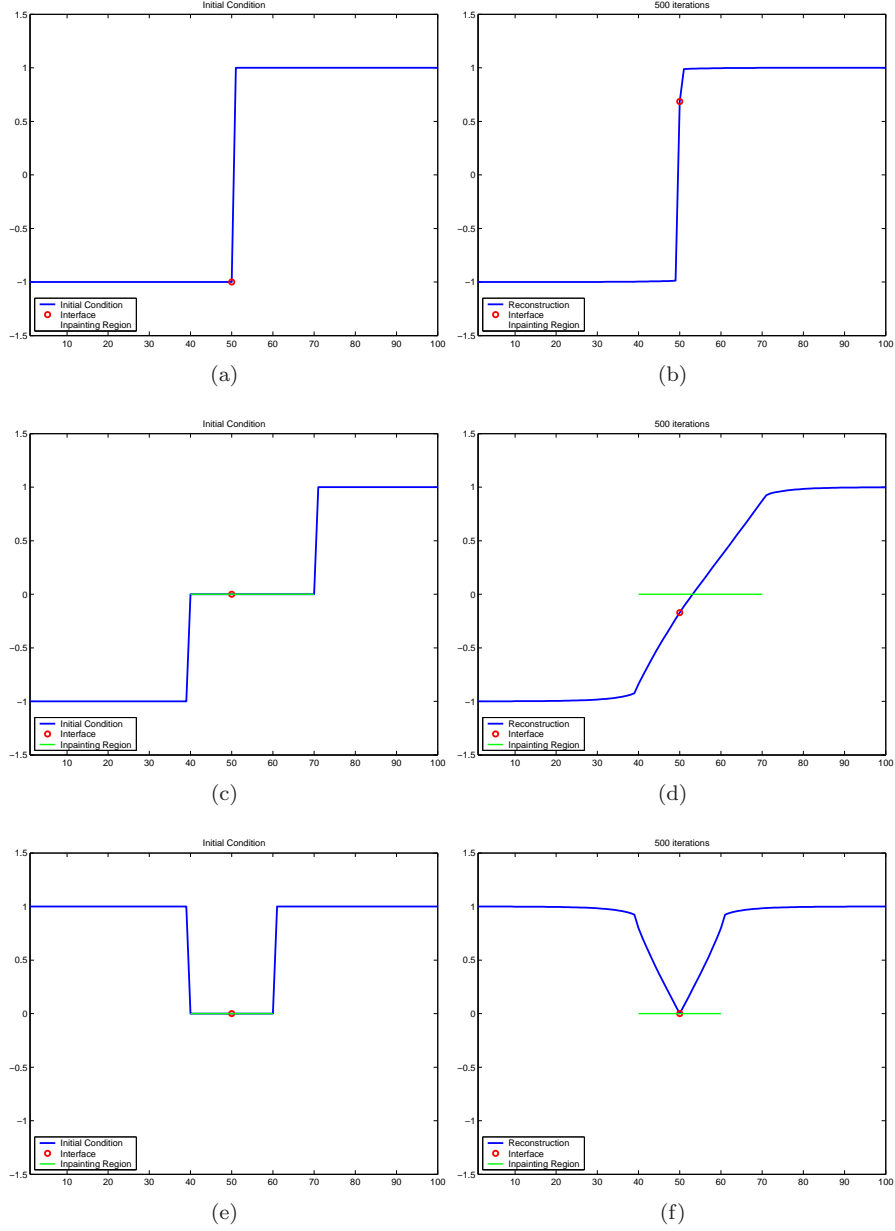

\begin{center}
\graphicspath{{./graphics/g_domaindecompeta_1D_0.01ep1dx0.5dt1lambda1eta500iter/}}
    \subfigure[]{\label{l1}\includegraphics[height=4.7cm]{g.eps}}
    \subfigure[]{\label{l2}\includegraphics[height=4.7cm]{u500.eps}}\\
\graphicspath{{./graphics/g2_domaindecompeta_1D_0.01ep1dx0.5dt1lambda1eta500iter/}}    
    \subfigure[]{\label{l3}\includegraphics[height=4.7cm]{g.eps}}
    \subfigure[]{\label{l4}\includegraphics[height=4.7cm]{u500.eps}}\\
\graphicspath{{./graphics/g3_domaindecompeta_1D_0.01ep1dx0.5dt1lambda1eta500iter/}}    
    \subfigure[]{\label{l5}\includegraphics[height=4.7cm]{g.eps}}
    \subfigure[]{\label{l6}\includegraphics[height=4.7cm]{u500.eps}}    
\end{center}    
\caption{Examples of $TV-L^2$ inpainting in 1D where the domain was split in two. (a)-(f): $\lambda=1$ and $\tau=1/2$}
\label{fignew}
\end{figure}

\begin{remark}\label{rmkfail}
Evaluating the soft thresholding operator at $u_2^{(n)}(N)-u_1^{(n+1)}(N-1)$ implies that we are treating  implicitly the computation of $u_1^{(n+1)}$ at the interface. Namely the interface condition can be read as
\begin{eqnarray*}
u_1^{(n+1)}(N) = & & u_2^{(n)}(N) - \Theta^{(n+1)} \cdot [u_2^{(n)}(N)-u_1^{(n+1)}(N-1) \\
& & - \sgn(u_2^{(n)}(N)-u_1^{(n+1)}(N-1))\cdot c_1^{(n)}(N) h],
\end{eqnarray*}
where 
$$
\Theta^{(n+1)} = \left \{
\begin{array}{ll} 
1,& |u_2^{(n)}(N)-u_1^{(n+1)}(N-1)|-c_1^{(n)}(N) h >0\\
0,& \mbox{ otherwise }.
\end{array}
\right .
$$ 
The solution of the implicit problem is not immediate and one may prefer to modify the situation in order to obtain  an explicit formulation by computing $S_{c_1^n(N) h}(u_2^{(n)}(N)-u_1^{(n)}(N-1))$ instead of $S_{c_1^n(N) h}(u_2^{(n)}(N)-u_1^{(n+1)}(N-1))$. The problem here is that, with this discretization, we cannot capture differences in the steepness of $u_1$ and $u_2$ at the interface because $u_1^{(n)}(N)=u_2^{(n)}(N)$ for all n. Indeed the condition $|u_2^{(n)}(N)-u_1^{(n)}(N-1)|-c_1^{(n)}(N) h >0$ is never satisfied and the interface becomes always a Dirichlet boundary condition. Even if we change the computation of $c_1^{(n)}(N)$ from $\sqrt{\epsilon^2+(u_1^{(n)}(N)-u_1^{(n)}(N-1))^2/h^2}$ to a forward difference $\sqrt{\epsilon^2+(u_1^{(n)}(N+1)-u_1^{(n)}(N))^2/h^2}$ (as it is indeed done in the numerical examples presented in Figure (\ref{fignew})) the method fails when the gradients are equal in absolute value on the left and the right side of the interface.
\end{remark}

We do not insist on trying to capture heuristically all the possible exceptions. We can expect that this approach to the problem may become even more deficient and more complicated to handle in 2D.
Instead, we want to apply the theory of the previous sections which allows to deal with the problem in a transparent way.

\subsubsection{The novel approach based on subspace corrections and oblique thresholding}
We want to implement the algorithm (\ref{schw_sp:it2}) for the minimization of $\mathcal J$. To solve its subiterations we compute the minimizer by means of oblique thresholding. Denote $u_2=u_2^{(n,M)}$, $u_1=u_1^{(n+1,\ell+1)}$, and $z=u_1^{(n+1,\ell)}+\pi_{V_1}T^*(g-Tu_2-T u_1^{(n+1,\ell)})$. We would like to compute the minimizer
$$
u_1 = \textrm{argmin}_{u \in V_1} \| u -z\|_{L^2(\Omega)}^2 + 2 \alpha |D(u+u_2)|(\Omega)
$$
by
$$
u_1 = (I- P_{\alpha K_{|D(\cdot)|(\Omega)}})(z+ u_2 -\eta)-u_2=\mathbb S_\alpha^{|D(\cdot)|(\Omega)}(z+ u_2 -\eta) -u_2,
$$
for any $\eta \in V_2$. It is known \cite{Ch} that $K_{|D(\cdot)|(\Omega)}$ is the closure of the set
$$
\left\{\textrm{div } \xi:\xi\in \left[C_c^1(\Omega)\right]^d, \left|\xi(x)\right|\leq 1 \quad\forall x\in\Omega\right\}.
$$
The element $\eta\in V_2$ is a limit of the corresponding fixed point iteration (\ref{fixptit}).
\\

In order to guarantee the concrete computability and the correctness of this procedure, we need to discretize the problem and approximate it in finite dimensions, compare Examples \ref{ex1}.3 and Remark \ref{rem1}.2.

In contrast to the approach of the previous section, where we used the discretization of the subdifferential to solve the subiterations, in the following we directly work with discrete approximations of the functional $\mathcal J$. In dimension $d=1$ we  consider vectors $u \in \mathcal H:= \mathbb{R}^N$, $u=(u_1,u_2,\ldots,u_N)$ with gradient $u_x\in \mathbb{R}^N$ given by
$$
(u_x)_i = \begin{cases}
u_{i+1}-u_i & \textrm{ if } i<N\\
0 & \textrm{ if } i=N,
\end{cases}
$$
for $i=1,\ldots,N$. In this setting, instead of minimizing 
$$
\mathcal J(u) := \| T u - g \|_{L^2{\Omega}}^2 + 2 \alpha |D(u)|(\Omega),
$$
we consider the discretized functional
$$
\mathcal J^\delta(u) := \sum_{1\leq i\leq N} \left(((Tu)_i-g_i)^2 + 2 \alpha |(u_x)_i|\right).
$$
To give a meaning to $(T u)_i$ we assume that $T$ is applied on the piecewise linear interpolant $\hat u$ of the vector $(u_i)_{i=1}^N$ (we will assume similarly for $d=2$).

In dimension $d=2$, the continuous image domain $\Omega=[a,b]\times[c,d]\subset\mathbb{R}^2$ is approximated by a finite grid $\left\{a=x_1<\ldots<x_N=b\right\}\times \left\{c=y_1<\ldots<x_M=d\right\}$ with equidistant step-size $h=x_{i+1}-x_i=\frac{b-a}{N}=\frac{d-c}{M}=y_{j+1}-y_j$ equal to $1$ (one pixel). The digital image $u$ is an element in $\mathcal H:=\mathbb{R}^{N\times M}$. We denote $u(x_i,y_j) = u_{i,j}$ for $i=1,\ldots,N$ and $j=1,\ldots, M$. The gradient $\nabla u$ is a vector in $\mathcal H\times \mathcal H$ given by forward differences
\begin{eqnarray*}
(\nabla u)_{i,j}=((\nabla_xu)_{i,j},(\nabla_yu)_{i,j}),
\end{eqnarray*}
with
\begin{eqnarray*}
& & (\nabla_x u)_{i,j} = \begin{cases}
u_{i+1,j}-u_{i,j} & \textrm{ if } i<N\\
0 & \textrm{ if } i=N,
\end{cases}\\
& & (\nabla_y u)_{i,j} = \begin{cases}
u_{i,j+1}-u_{i,j} & \textrm{ if } j<M\\
0 & \textrm{ if } j=M,
\end{cases}
\end{eqnarray*}
for $i=1,\ldots,N$, $j=1,\ldots, M$. The discretized functional in two dimensions is given by
$$
\mathcal J^\delta(u) := \sum_{1\leq i,j\leq N} \left(((Tu)_{i,j}-g_{i,j})^2 + 2 \alpha |(\nabla u)_{i,j}|\right),
$$
with $\left|y\right|=\sqrt{y_1^2+y_2^2}$ for every $y=(y_1,y_2)\in\mathbb{R}^2$.

For the definition of the set $K_{|D(\cdot)|(\Omega)}$ in finite dimensions we further introduce a discrete divergence in one dimension $\nabla\cdot : \mathcal H\rightarrow \mathcal H$ (resp. $\nabla\cdot : \mathcal H\times \mathcal H\rightarrow \mathcal H$ in two dimensions) defined, by analogy with the continuous setting, by $\nabla\cdot = -\nabla^*$ ($\nabla^*$ is the adjoint of the gradient $\nabla$). That is, the discrete divergence operator is given by backward differences, in one dimension by
$$
(\nabla\cdot p)_i = \begin{cases}
p_i-p_{i-1} & \textrm{ if } 1<i<N\\
p_i & \textrm{ if } i=1\\
-p_{i-1} & \textrm{ if } i=N,
\end{cases}
$$
and, respectively, in two dimensions by
\begin{eqnarray*}
(\nabla\cdot p)_{ij} &=& \begin{cases}
(p^x)_{i,j}-(p^x)_{i-1,j} & \textrm{ if } 1<i<N\\
(p^x)_{i,j} & \textrm{ if } i=1\\
-(p^x)_{i-1,j} & \textrm{ if } i=N
\end{cases}
\\
& & + 
\begin{cases}
(p^y)_{i,j}-(p^y)_{i,j-1} & \textrm{ if } 1<j<M\\
(p^y)_{i,j} & \textrm{ if } j=1\\
-(p^y)_{i,j-1} & \textrm{ if } j=M,
\end{cases}
\end{eqnarray*}
for every $p=(p^x,p^y)\in \mathcal H\times \mathcal H$.

With these definitions the set $K_{\|(\cdot)_x\|_{\ell_1^N}}$ in one dimension is given by 
\begin{eqnarray*}
\left\{\nabla\cdot p: p\in \mathcal H, \left|p_i\right|\leq 1\;\forall i=1,\ldots, N\right\}.
\end{eqnarray*}
and in two dimensions $K_{\|\nabla(\cdot)\|_{\ell_1^{N\times M}}}$ is given by
\begin{eqnarray*}
\left\{\nabla\cdot p: p\in \mathcal H\times \mathcal H, \left|p_{i,j}\right|\leq 1\;\forall i=1,\ldots, N\textrm{ and } j=1,\ldots,M \right\}.
\end{eqnarray*}

To highlight the relationship between the continuous and discrete setting we introduce a step-size $h\sim 1/N$ in 1D ($h\sim \min\{1/N,1/M\}$ in 2D respectively) in the discrete definition of $\mathcal{J}$ by defining a new functional $\mathcal J_h^\delta$ equal to $h$ times the expression $\mathcal J^\delta$ above. One can show that as $h\rightarrow 0$, $\mathcal J_h^\delta$ $\Gamma-$ converges to the continuous functional $\mathcal J$, see \cite{Ch}. In particular, piecewise linear interpolants $\hat u_h$ of the minimizers of the discrete functional $\mathcal J_h^\delta$ do converge to minimizers of $\mathcal J$. This observation clearly justifies our discretization approach.

For the computation of the projection in the oblique thresholding we can use an algorithm proposed by Chambolle in \cite{Ch}.  In two dimensions the following semi-implicit gradient descent algorithm is given to approximate $P_{\alpha K_{|\nabla (\cdot)|(\Omega)}}(g)$:
\begin{quote} Choose $\tau>0$, let $p^{(0)}=0$ and, for any $n\geq 0$, iterate
\begin{eqnarray*}
p_{i,j}^{(n+1)} = p_{i,j}^{(n)} + \tau \left((\nabla(\nabla\cdot p^{(n)}-g/\alpha))_{i,j} - \left|(\nabla(\nabla\cdot p^{(n)}-g/\alpha))_{i,j}\right|p_{i,j}^{(n+1)}\right),
\end{eqnarray*}
so that
\begin{eqnarray}\label{chprojit}
p_{i,j}^{(n+1)} = \frac{p_{i,j}^{(n)} + \tau (\nabla(\nabla\cdot p^{(n)}-g/\alpha))_{i,j}}{1+\tau \left|(\nabla(\nabla\cdot p^{(n)}-g/\alpha))_{i,j}\right|}.
\end{eqnarray}\end{quote}

For $\tau\leq 1/8$ the iteration $\alpha\nabla\cdot p^{(n)}$ converges to $P_{\alpha K_{|\nabla (\cdot )|(\Omega)}}(g)$ as $n\rightarrow\infty$ (compare \cite[Theorem 3.1]{Ch}). 

For $d=1$  a similar algorithm is given:
\begin{quote}We choose $\tau>0$, let $p^{(0)}=0$ and for any $n\geq 0$,
\begin{eqnarray}\label{chprojit1}
p_i^{(n+1)} = \frac{p_i^{(n)} + \tau ((\nabla\cdot p^{(n)}-g/\alpha)_x)_i}{1+\tau \left|((\nabla\cdot p^{(n)}-g/\alpha)_x)_i\right|}.
\end{eqnarray}\end{quote}
In this case the convergence of $\alpha\nabla\cdot p^{(n)}$ to the corresponding projection as $n\rightarrow\infty$  is guaranteed for $\tau\leq 1/4$.

\subsubsection{Domain decompositions}
\label{tricks}

In one dimension the domain $\Omega=[a,b]$ is split into two intervals $\Omega_1=[a,\left\lceil \frac{N}{2}\right\rceil]$ and $\Omega_2=[\left\lceil \frac{N}{2}\right\rceil+1,b]$. The interface $\partial\Omega_1\cap\partial\Omega_2$ is located between $i=\left\lceil N/2\right\rceil$ in $\Omega_1$ and $i=\left\lceil N/2\right\rceil+1$ in $\Omega_2$. In two dimensions the domain $\Omega=[a,b]\times[c,d]$ is split in an analogous way with respect to its rows. In particular we have $\Omega_1=[a,\left\lceil \frac{N}{2}\right\rceil]\times[c,d]$ and $\Omega_2=[\left\lceil \frac{N}{2}\right\rceil+1,b]\times[c,d]$, compare Figure \ref{figill}. The splitting in more than two domains is done similarly: 
\begin{quote}
Set $\Omega=\Omega_1\cup\ldots\cup\Omega_\mathcal N$, the domain $\Omega$ decomposed into $\mathcal N$ disjoint domains $\Omega_i$, $i=1,\ldots, \mathcal N$. Set $s=\left\lceil N/\mathcal N\right\rceil$. Then
\begin{eqnarray*}
& & \Omega_1=[1,s]\times[c,d]\\
& & \textrm{for } i=2:\mathcal N-1\\
& & \quad\Omega_i=[(i-1)s+1,is]\times[c,d]\\
& & \textrm{end}\\
& & \Omega_\mathcal N=[(\mathcal N-1)s+1,N]\times[c,d].
\end{eqnarray*}
\end{quote}

\begin{center}
\begin{figure}[h!]
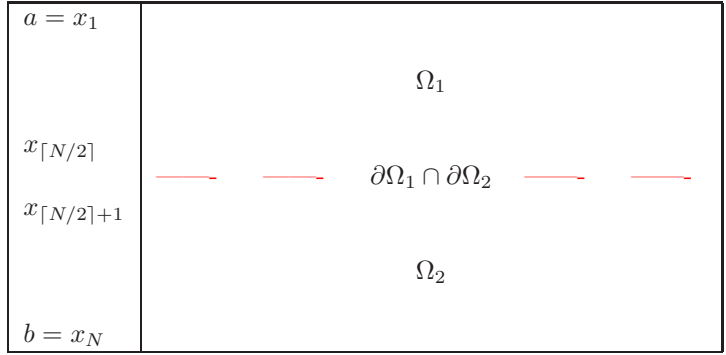
\begin{center}
\begin{tabular}{|l| p{1 cm} p{1 cm} c p{1 cm} p{1 cm} |}
\hline
$a=x_1$ & & & & &\\
 & & & & &\\
 & & & $\Omega_1$ & &\\
 & & & & &\\
$x_{\left\lceil N/2\right\rceil}$ & & & & &\\
 &\textcolor{red}{ ------- }&\textcolor{red}{ ------- }& $\partial\Omega_1\cap\partial\Omega_2$ &\textcolor{red}{ ------- }&\textcolor{red}{ ------- }\\
$x_{\left\lceil N/2\right\rceil+1}$ & & & & &\\ 
 & & & & &\\
 & & & $\Omega_2$ & &\\
 & & & & &\\
$b=x_N$ & & & & &\\
\hline
\end{tabular}\end{center}
\caption{Decomposition of the discrete image in two domains $\Omega_1$ and $\Omega_2$ with interface $\partial\Omega_1\cap\partial\Omega_2$}
\label{figill}
\end{figure}
\end{center}

To compute the fixed point $\eta$ of \eqref{fixpt} in an efficient way we make the following considerations, which allow to restrict the computation to a relatively small stripe around the interface.
For $u_2 \in V_2^\psi$ and $z\in V_1$ a minimizer $u_1$ is given by
$$
u_1 = \textrm{argmin}_{u \in V_1} \| u -z\|_{L^2(\Omega)}^2 + 2 \alpha |D(u+u_2)|(\Omega).
$$
We further decompose  $\Omega_2=\hat{\Omega}_2\cup(\Omega_2\setminus\hat{\Omega}_2)$ with $\partial\hat{\Omega}_2\cap\partial\Omega_1=\partial\Omega_2\cap\partial\Omega_1$, where  $\hat{\Omega}_2\subset\Omega_2$ is a neighborhood stripe around the interface $\partial {\Omega}_2\cap\partial\Omega_1$, as illustrated in  Figure \ref{figill2}. By using the splitting of the total variation \eqref{amb} we can restrict the problem to  an equivalent minimization where the total variation is only computed in $\Omega_1\cup\hat{\Omega}_2$. Namely, we have
$$
u_1 = \textrm{argmin}_{u \in V_1} \| u -z\|_{L^2(\Omega)}^2 + 2 \alpha |D(u+u_2)|(\Omega_1\cup\hat{\Omega}_2).
$$

\begin{center}
\begin{figure}[h!]
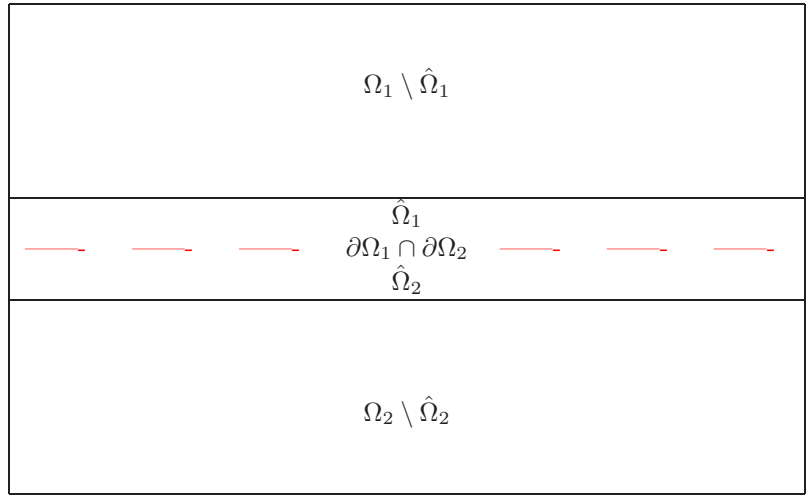
\begin{center}
\begin{tabular}{|p{1 cm}  p{1 cm} p{1 cm} c p{1 cm} p{1 cm} p{1 cm}|}
\hline
 & & & & & &\\
 & & & & & &\\

 & & & $\Omega_1\setminus\hat{\Omega}_1$  & & &\\
 & & & & & &\\
 & & & & & &\\
 & & & & & &\\  
 \hline
 & & &$\hat{\Omega}_1$ & & &\\ 
 \textcolor{red}{ ------- } & \textcolor{red}{ ------- } & \textcolor{red}{ ------- } & $\partial\Omega_1\cap\partial\Omega_2$ & \textcolor{red}{ ------- } & \textcolor{red}{ ------- }& \textcolor{red}{ ------- }\\
 & & & $\hat{\Omega}_2$ & & &\\  
 \hline
 & & & & & &\\
 & & & & & &\\
 & & & & & &\\ 
 & & & $\Omega_2\setminus\hat{\Omega}_2$ & & &\\
 & & & & & &\\
 & & & & & &\\
\hline
\end{tabular}\end{center}
\caption{Computation of $\eta$ only in the stripe $\hat{\Omega}_1\cup\hat{\Omega}_2$.}
\label{figill2}
\end{figure}
\end{center}

Hence, for the computation of the fixed point $\eta\in V_2$, we need to carry out the iteration $\eta^{(m+1)} = \pi_{V_2}P_{\alpha K_{|D(\cdot)|(\Omega)}}(\eta^{(m)}-z+u_2)$ only in $\Omega_1\cup\hat{\Omega}_2$. By further observing that $\eta$ will be supported only in $\Omega_2$, i.e. $\eta(x)=0$ in $\Omega_1$, we may additionally restrict the fixed point iteration on the relatively small stripe $\hat{\Omega}_1\cup\hat{\Omega}_2$, where $\hat{\Omega}_1\subset\Omega_1$ is an neighborhood around the interface from the side of $\Omega_1$. Although the computation of $\eta$ restricted to $\hat{\Omega}_1\cup\hat{\Omega}_2$ is not equivalent to the computation of $\eta$ on whole $\Omega_1\cup\hat{\Omega}_2$, the produced errors are in practice negligible, because of the Neumann boundary conditions involved in the computation of $P_{\alpha K_{|\nabla (\cdot)|(\Omega_1\cup\hat{\Omega}_2)}}$. Symmetrically, one operates on the minimizations on $\Omega_2$.

\medskip
\begin{figure}
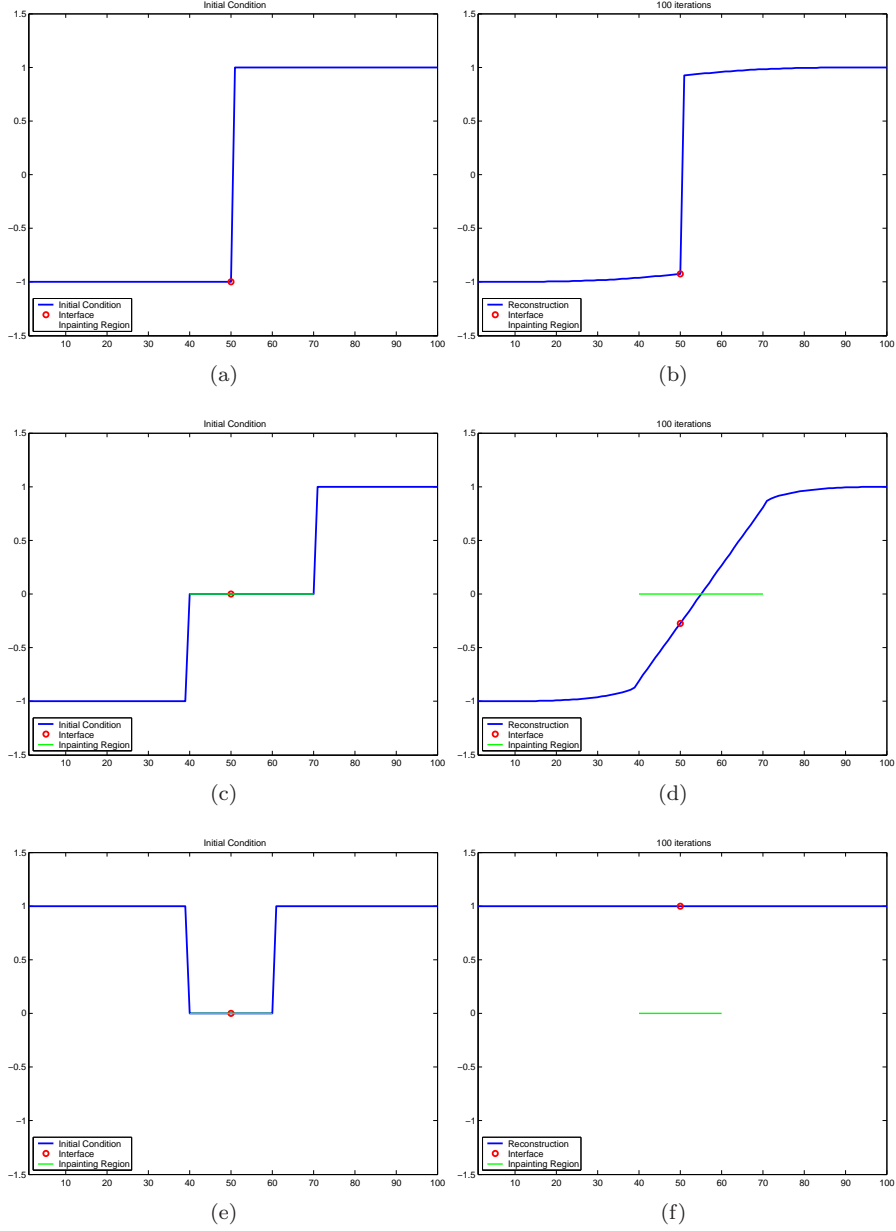

\begin{center}
\graphicspath{{./graphics/g_dode_sinsplit_1D_acc_lambda_0_timestepsize_0.25_end_100_subit5_itetamax10_projerr3_stripe3/}}
    \subfigure[]{\label{mm1}\includegraphics[height=4.7cm]{g.eps}}
    \subfigure[]{\label{mm2}\includegraphics[height=4.7cm]{u.eps}}\\
\graphicspath{{./graphics/g2_dode_sinsplit_1D_acc_lambda_0_timestepsize_0.25_end_100_subit5_itetamax10_projerr3_stripe10/}}    
    \subfigure[]{\label{mm3}\includegraphics[height=4.7cm]{g.eps}}
    \subfigure[]{\label{mm4}\includegraphics[height=4.7cm]{u.eps}}\\
\graphicspath{{./graphics/g3_dode_sinsplit_1D_acc_lambda_0_timestepsize_0.25_end_100_subit5_itetamax10_projerr3_stripe10/}}    
    \subfigure[]{\label{mm5}\includegraphics[height=4.7cm]{g.eps}}
    \subfigure[]{\label{mm6}\includegraphics[height=4.7cm]{u.eps}}    
\end{center}    
\caption{(a)-(f): Examples of the domain decomposition method for $TV-L^2$ denoising/inpainting in 1D where the domain was split in two domains with $\alpha=1$ and $\tau=1/4$}
\label{fignew1}
\end{figure}

\subsubsection{Numerical experiments in one and two dimensions}

We shall present  numerical results in one and two dimensions for the algorithm in \eqref{schw_sp:it2}, and discuss them with respect to the choice of parameters.

In one dimension we consider the same three signals already discussed for the ``naive'' approach in Figure \ref{fignew}. In the left column of Figure \ref{fignew1} we report again the one dimensional signals. The right column shows the result of the application of the domain decomposition method \eqref{schw_sp:it2} for total variation minimization. The support of the signals is split in two intervals. The interface developed by the two intervals is marked by a red dot. In all three examples we fixed $\alpha=1$ and $\tau=1/4$. The first example \ref{m1}-\ref{m2} shows a step function, which has its step directly at the interface of the two intervals. The total variation minimization   \eqref{schw_sp:it2} is applied with $T=I$. This example confirms that jumps are preserved at the interface of the two domains. The second and third example \ref{mm3}-\ref{mm6} present the behaviour of the algorithm when interpolation across the interface is performed. In this case the operator $T$ is given by the multiplier $T=1_{\Omega\setminus D}$, where $D$ is an interval containing the interface point. In contrast to the performance of the interpolation of the ``naive'' approach for the third example, Figure \ref{l5}-\ref{l6}), the new approach solves the interpolation across the interface correctly, see Figure \ref{mm5}-\ref{mm6}.
\begin{figure}
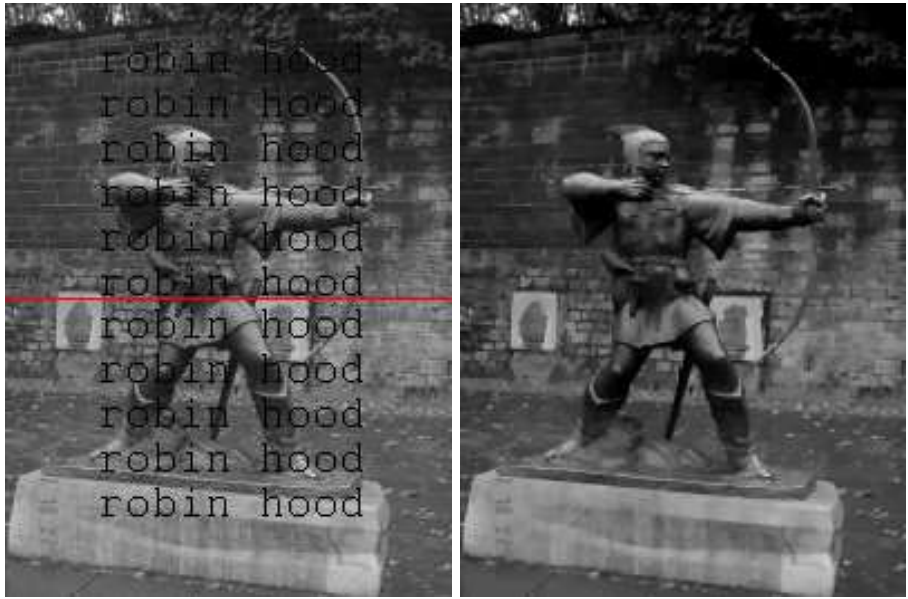

\begin{center}
\graphicspath{{./graphics/robin_hood_des1.png_dode_sinsplit_acc_lambda-2_dt0.25_stripe3_subit5_itetamax10_projerr3/}}
\includegraphics[height=8cm]{uinpaint.eps}\includegraphics[height=8cm]{u.eps}
\end{center}    
\caption{An example of $TV-L^2$ inpainting in where the domain was split in two with $\alpha=10^{-2}$ and $\tau=1/4$}
\label{fignewD2}
\end{figure}

\begin{figure}
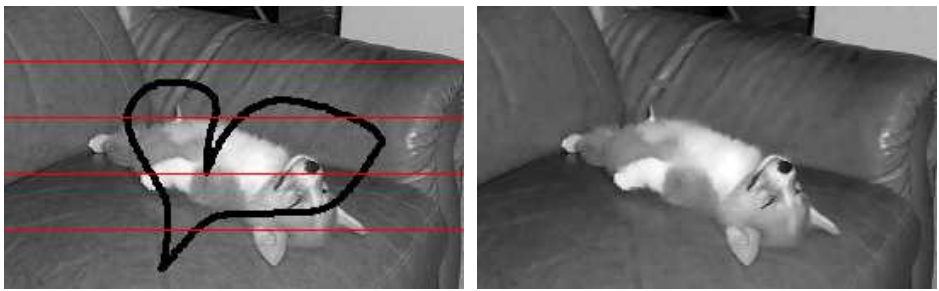

\begin{center}
\graphicspath{{./graphics/sleeping_dog_des2.png_dodemore_sinsplit_acc_D5_stripe3_lambda_-2_dt0.125_subit3_itetamax10/}}
\includegraphics[height=3.8cm]{uinpaint.eps} \includegraphics[height=3.8cm]{u.eps}
\end{center}    
\caption{An example of $TV-L^2$ inpainting in where the domain was splitted in five with $\alpha=10^{-2}$ and $\tau=1/4$}
\label{fignew2D3}
\end{figure}

Inpainting results for the two dimensional case are shown in Figures \ref{fignewD2}-\ref{fignew2D3}. The interface is here marked by a red line in the given image. In the first example in Figure \ref{fignewD2} the domain is split in two subdomains, in the second example in Figure \ref{fignew2D3} the domain is split in five subdomains. The Lagrange multiplier $\alpha>0$ is chosen $10^{-2}$. The time-step for the computation of $P_{\alpha K_{|\nabla(\cdot)|(\Omega)}}$ is chosen $\tau=1/4$. The examples confirm the correct reconstruction of the image at the interface, preserving both continuities and discontinuities as wanted.
Despite the fact that Theorem \ref{weak-conv} does not guarantee that the algorithm in \eqref{schw_sp:it2} can converge to a minimizer of $\mathcal J$ (unless one of the conditions in (iv) holds), it seems that for total variation minimization the result is always rather close to the expected minimizer.
\\

Let us now discuss the choice of the different parameters.
As a crucial issue in order to compute the solution at the interface $\partial\Omega_1\cap\partial\Omega_2$ correctly, one has to pay attention to the accuracy up to which the projection $P_{\alpha K_{|\nabla(\cdot)|(\Omega)}}$ is approximated and to the width of the stripe $\hat{\Omega}_1\cup\hat{\Omega}_2$ for the computation of $\eta$.
The alternating iterations \eqref{schw_sp:it2} in practice are carried out with $L=M=5$ inner iterations. The outer iterations are carried out until the error $\left|J(u^{(n+1)})-J(u^{(n)})\right|$ is of order $\mathcal{O}(10^{-10})$. The fixed point $\eta$ is computed by iteration \eqref{fixptit} in maximal $10$ iterations with initialization $\eta^{(0)}_n=0$ when $n=1$ and $\eta^{(0)}_{n+1}=\eta_n$, the $\eta$ computed in the previous iteration, for $n>1$. For the computation of the projection $P_{\alpha K_{|\nabla(\cdot)|(\Omega)}}$ by Chambolle's algorithm \eqref{chprojit} we choose $\tau=1/4$. Indeed Chambolle points out in \cite{Ch} that, in practice, the optimal constant for the stability and convergence of the algorithm is not $1/8$ but $1/4$. Further if the derivative along the interface is high, i.e., if there is a step along the interface, one has to be careful concerning the accuracy of the computation for the projection. The stopping criterion for the iteration \eqref{chprojit} consists in checking that the maximum variation between $p_{i,j}^n$ and $p_{i,j}^{n+1}$ is less than $10^{-3}$. With less accuracy, artifacts on the interface can appear. This error tolerance may need to be further decreased for $\alpha>0$ very large. Furthermore, the size of the stripe varies between $6$ and $20$ pixels also depending on the size of $\alpha>0$, and if either inpainting is carried out via the interface or not (e.g., the second and third example in Figure \ref{fignew1} failed in reproducing the interface correctly with a stripe of size $6$ but computed it correctly with a stripe of size $20$).

\subsection{Accelerated sparse recovery algorithms based on $\ell_1$-minimization}
\label{accel}
In this section we are concerned with applications of the algorithms described in the previous sections to the case where $\Lambda$ is a countable index set, $\mathcal H = \ell_2(\Lambda)$, and $\psi(u) = \| u\|_{\ell_{1}(\Lambda)}:= \sum_{\lambda \in \Lambda} | u_\lambda|$, compare Examples \ref{ex1}.2. In this case we are interested to the minimization of the functional
\begin{equation}
\label{sparse}
\mathcal J(u) := \| T u - g \|_{\ell_2(\Lambda)}^2 + 2 \alpha \|u\|_{\ell_1}.
\end{equation}
As already mentioned, iterative algorithms of the type (\ref{eq1}) can make the job, where $\mathbb S_\alpha^{\|\cdot\|_{\ell_1}} =  \mathbb S_\alpha$ is the soft-thresholding. Unfortunately, despite its simplicity which makes it very attractive to users, this algorithm does not perform very well.
For this reason the ``domain decomposition'' algorithm (\ref{schw_sp:it}) was proposed in \cite{fo07}, and there we proved its effectiveness in accelerating the convergence and we provided its parallelization. Here the domain is the label set $\Lambda$ which is disjointly decomposed into $\Lambda = \Lambda_1 \cup \Lambda_2$. This decomposition produces an orthogonal splitting of $\ell_2(\Lambda)$ into the subspaces  $V_i = \ell_2^{\Lambda_i}(\Lambda) :=\{u \in \ell_2(\Lambda): \supp(u) \subset \Lambda_i\}$, $i=1,2$.
We want to generalize this particular situation to an arbitrary orthogonal decomposition:\\
 Let $Q$ be an orthogonal operator on $\ell_2(\Lambda)$. With this operator we denote $Q_{\Lambda_i} := Q \pi_{\ell_2^{\Lambda_i}(\Lambda)}$.  Finally we can define $V_i := Q_{\Lambda_i} \ell_2(\Lambda)$ for $i=1,\dots,\mathcal N$. In particular, we can consider multiple subspaces, i.e., $\mathcal N \geq 2$, since the algorithms and their analysis presented in the previous sections can be easily generalized to these cases, see \cite[Section 6]{fo07}. For simplicity we assume that the subspaces have equal dimensions when $\dim \mathcal H < \infty$.
 Clearly the orthogonal projection onto $V_i$ is given by $\pi_{V_i} = Q_{\Lambda_i} Q^*_{\Lambda_i}$. Differently from the domain decomposition situation for which $Q=I$ and 
\begin{equation}
\label{nicesplit}
\| \sum_{i=1}^{\mathcal N} \pi_{V_i} u  \|_{\ell_1} = \sum_{i=1}^{\mathcal N}  \| \pi_{V_i} u \|_{\ell_1},
\end{equation}
for an arbitrary splitting, i.e., for  $Q\neq I$,  (\ref{nicesplit}) is not guaranteed to hold. Hence, an algorithm as in (\ref{schw_sp:it}) cannot anymore be applied and one has to use (\ref{schw_sp:it2}) or (\ref{schw_sp:it4}) instead.
In finite dimensions there are several ways to compute suitable operators $Q$. The constructions we consider in our numerical examples are given by $Q$ as the orthogonalization of a random matrix $\tilde Q$, e.g., via Gram-Schmidt, or the orthogonal matrix $Q=V$ provided by the singular value decomposition of $T = U D V^*$. 
Of course, for very large matrices $T$, the computation of the SVD is very expensive. In these cases, one may want to use the more efficient strategy proposed in \cite{ruve}, where $Q$ is constructed by computing the SVD of a relatively small submatrix of $T$ generated by random sampling.

The numerical examples presented in the following, refer to applications of the algorithms for the minimization of $\mathcal J$ where the operator $T$ is a random matrix $200 \times 40$ with Gaussian entries.

\subsubsection{Discussion on the convergence properties of the algorithm}
It is stated in the Theorems \ref{weak-conv} and \ref{weak-conv2} that the algorithms (\ref{schw_sp:it2}) or (\ref{schw_sp:it4}) may not converge  to a minimizer of $\mathcal J$ for an arbitrary orthogonal operator $Q$, while, in reason of (\ref{nicesplit}) and  condition (a) in Theorem \ref{weak-conv} (iv), such convergence is guaranteed for $Q=I$. 
\begin{center}
\begin{figure}[ht]

\hbox to \hsize {\hfill \epsfig{file=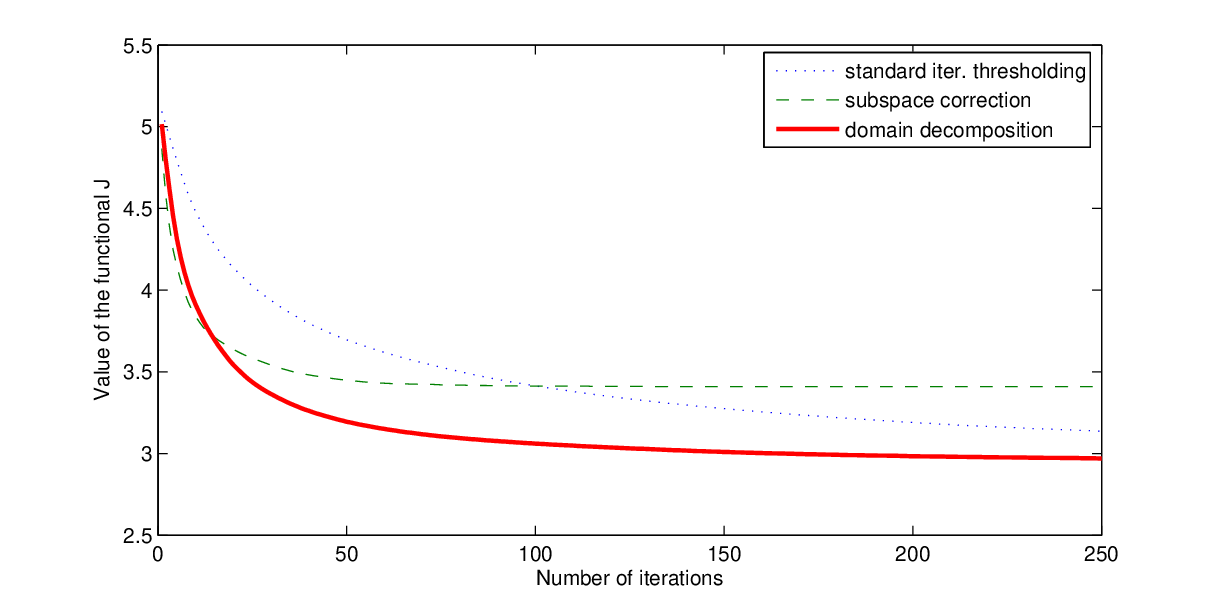,height=5cm} \hfill}
\caption{We show the application of the algorithms (\ref{ex1}), (\ref{schw_sp:it2}), and (\ref{schw_sp:it}), for the minimization of $\mathcal J$ where $T$ is a random matrix $200 \times 40$ with Gaussian entries. We considered in this example $\mathcal N=5$ subspaces $V_i$, $\alpha=0.005$, $30$ external iterations and $30$ internal iterations for the minimization on each $V_i$, $i=1,\dots,5$. We  fixed a maximal number of $20$ iterations in the approximate computation of the auxiliary $\eta$'s in (\ref{fixptit}). While (\ref{schw_sp:it}) converges to a minimizer of $\mathcal J$, this is not the case for (\ref{schw_sp:it2}).} \label{figure1}
\end{figure}
\end{center}

 \begin{center}
\begin{figure}[ht]

\hbox to \hsize {\hfill \epsfig{file=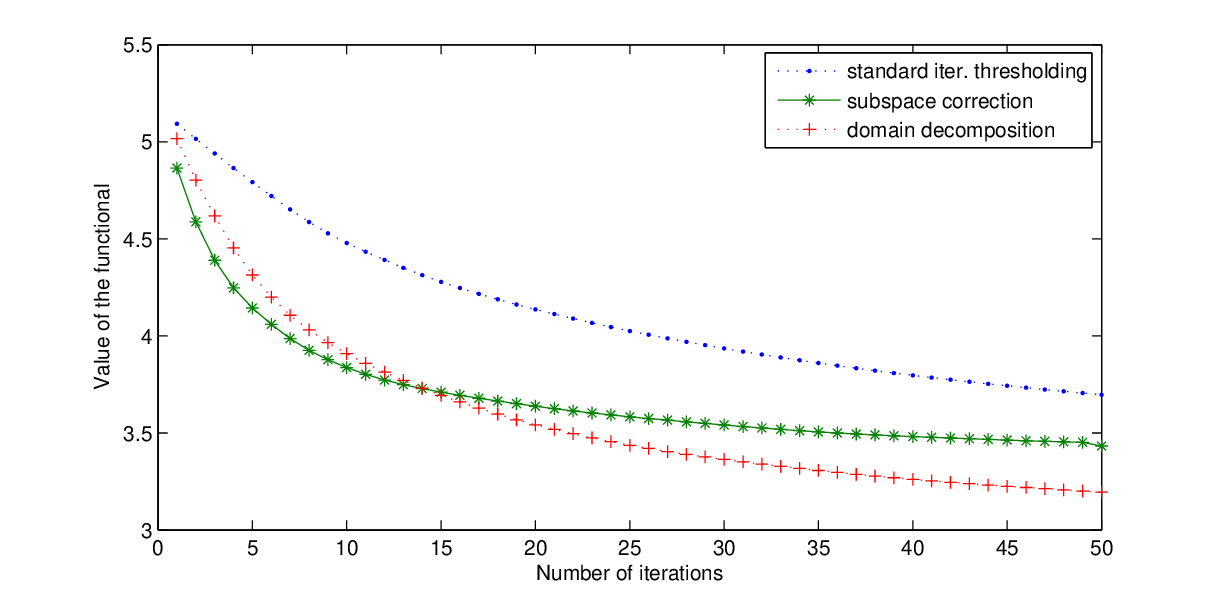,height=5cm} \hfill}
\caption{We show the application of the algorithms (\ref{ex1}), (\ref{schw_sp:it2}), and  (\ref{schw_sp:it}), for the minimization of $\mathcal J$ where $T$ is a random matrix $200 \times 40$ with Gaussian entries. We considered in this example $\mathcal N=5$ subspaces $V_i$, $\alpha=0.005$, $30$ external iterations and $30$ internal iterations for the minimization on each $V_i$, $i=1,\dots,5$. We  fixed a maximal number of $20$ iterations in the approximate computation of the auxiliary $\eta$'s in (\ref{fixptit}). In the first few iterations (\ref{schw_sp:it2}) converges faster than (\ref{schw_sp:it}).} \label{figure2}
\end{figure}
\end{center}
In Figure \ref{figure1} we can illustrate the different convergence behavior for $Q \neq I$ and for $Q=I$, by a comparison with the standard iterative thresholding algorithm (\ref{eq1}). Nevertheless, in several situations the computed solution due to (\ref{schw_sp:it2}) or (\ref{schw_sp:it4}) for $Q \neq I$ is very close to the wanted minimizer, especially for $\alpha>0$ relatively small.
 Moreover, it is important to observe that the choice of a suitable $Q$, for example the one provided by the singular value decomposition of $T$, does accelerate the convergence in the very first iterations, as we illustrate in Figure \ref{figure2}. In particular, within the first few iterations, most of the important information on the support of the minimal solution $u$ is recovered. This explains the rather significant acceleration of the convergence shown in Figure \ref{figure3} obtained by combining few initial iterations of the algorithm  (\ref{schw_sp:it2}) for the choice of $Q=V$ with successive iterations where the choice is switched to $Q=I$ in order to ensure convergence to minimizers of $\mathcal J$. This combined strategy proved to be extremely efficient and it is the one we consider in the rest of our discussion.
 \begin{center}
\begin{figure}[ht]

\hbox to \hsize {\hfill \epsfig{file=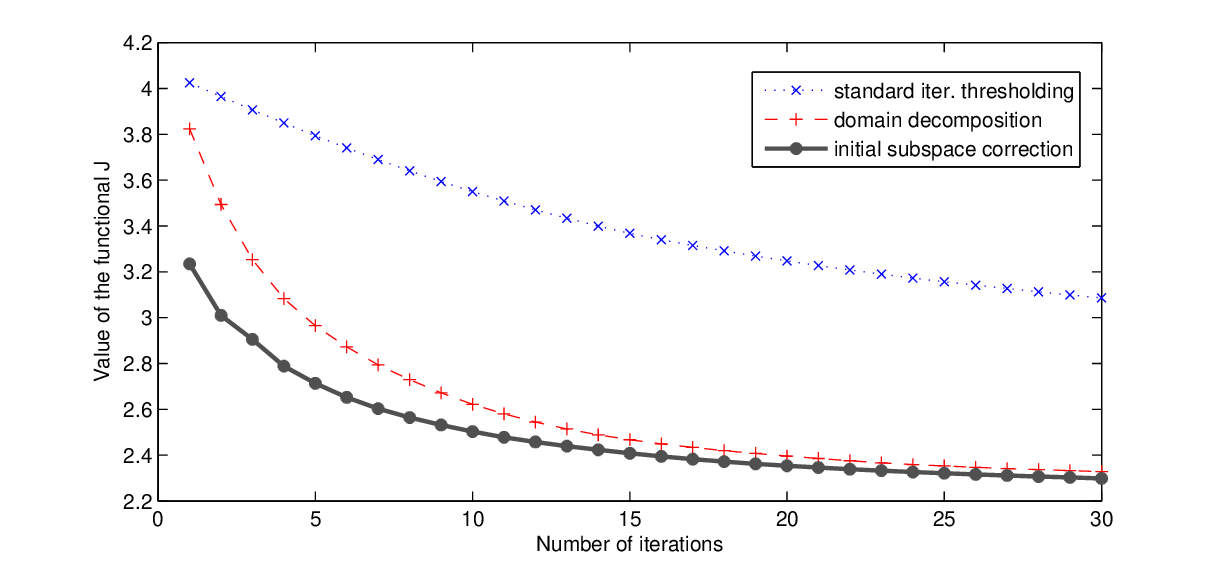,height=5cm} \hfill}
\caption{We show the application of the algorithms (\ref{ex1}), (\ref{schw_sp:it}), and (\ref{schw_sp:it2}) where the first 4 external iterations are performed with $Q=V$, followed by iterations where $Q=I$.  Again the minimization of $\mathcal J$ is performed assuming that $T$ is a random matrix $200 \times 40$ with Gaussian entries and the same parameters as in the previous figures. The starting acceleration due to the initial choice of $Q=V$ allows to recover sufficient information on the support of the sparse minimal solution, so that the following iterations for $Q=I$ do already perform significantly better than (\ref{schw_sp:it}).} \label{figure3}
\end{figure}
\end{center}

 \begin{center}
\begin{figure}[ht]

\hbox to \hsize {\hfill \epsfig{file=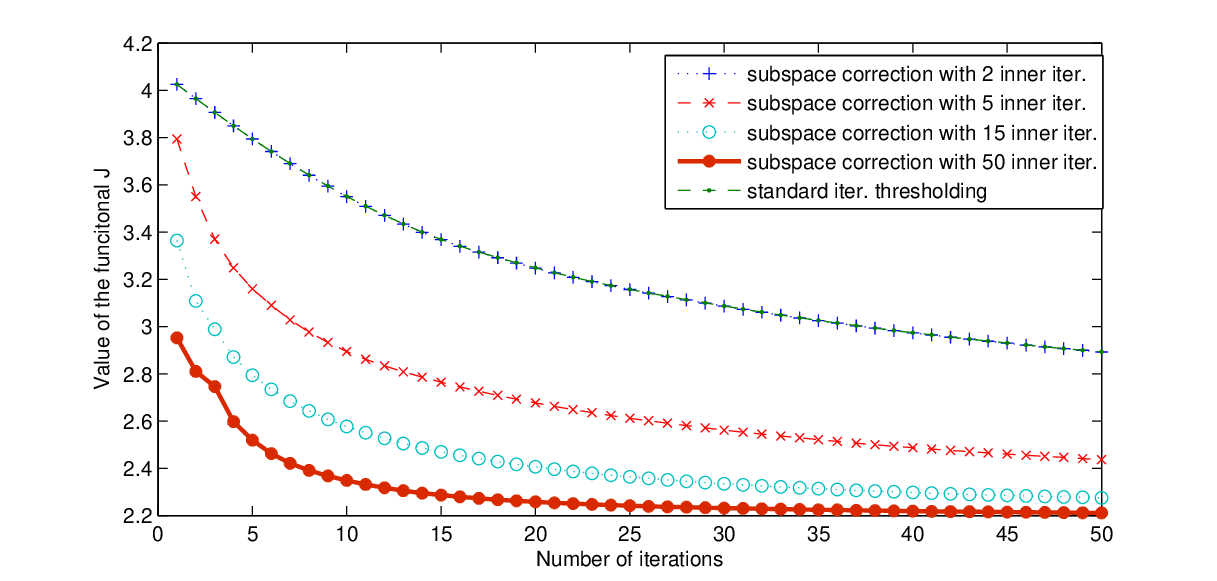,height=5cm} \hfill}
\caption{We show the application of the algorithms (\ref{ex1}) and (\ref{schw_sp:it2}) (with the switching from $Q=V$ to $Q=I$ as in the previous figure) for the minimization of $\mathcal J$ where $T$ is a random matrix $200 \times 40$ with Gaussian entries. We considered in this application of (\ref{schw_sp:it2})  $\mathcal N=10$ subspaces $V_i$, $\alpha=0.005$, $50$ external iterations and an increasing number of inner iterations for the minimization on each $V_i$, $i=1,\dots,5$. We  fixed a maximal number of $20$ iterations in the approximate computation of the auxiliary $\eta$'s in (\ref{fixptit}). We can observe that by increasing the number of inner iterations we can significantly improve the rate of convergence with respect to the external iterations.} \label{figure4}
\end{figure}
\end{center}

We developed further experiments for the evaluation of the behavior of the algorithm (\ref{schw_sp:it2}) with respect to other parameters, in particular the number of inner iterations for the minimization on each $V_i$, $i=1,\dots,\mathcal N$, and  the number $\mathcal N$ of subspaces. In Figure \ref{figure4} we show that by increasing the number of inner iterations we improve significantly the convergence with respect to the outer iterations. Of course, the improvement due to an increased number of inner iterations corresponds also to an increased computational effort. In order to counter-balance this additional cost one may consider a larger number of subspaces, that in turns implies a smaller dimension of each subspace. Indeed, note that inner iterations on subspaces with smaller dimension require a much less number of algebraic operations.
In Figure \ref{figure5} we show that by increasing the number of subspaces and, correspondingly, the number of inner iterations we do keep improving the convergence significantly. Hence, the \emph{parallel} algorithm (\ref{schw_sp:it4}) adapted to a large number of subspaces performs very fast as soon as the inner iterations are also increased correspondingly.

 \begin{center}
\begin{figure}[ht]

\hbox to \hsize {\hfill \epsfig{file=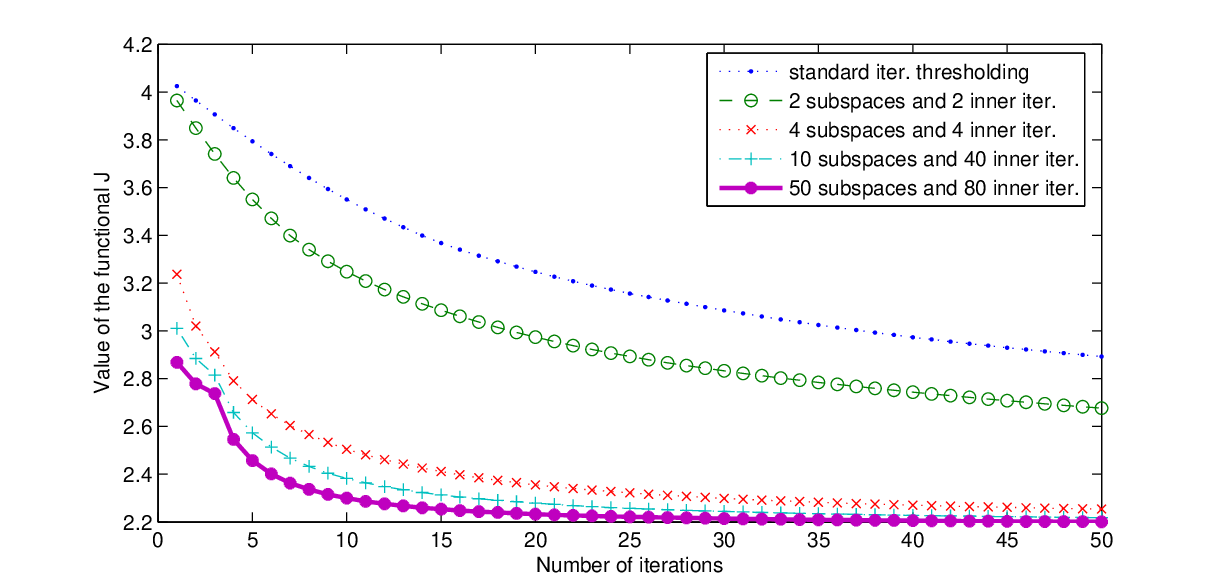,height=5cm} \hfill}
\caption{{We show the application of the algorithms (\ref{ex1}) and (\ref{schw_sp:it2}) (with the switching from $Q=V$ to $Q=I$ after $4$ external iterations) for the minimization of $\mathcal J$ where $T$ is a random matrix $200 \times 40$ with Gaussian entries. We considered in this application of (\ref{schw_sp:it2})  an increasing number $\mathcal N=2,4,10,50$ of subspaces $V_i$ and correspondingly an increasing number, $2,4,40,80$, of inner iterations. Again we fixed $\alpha=0.005$, $50$ external iterations, and a maximal number of $20$ iterations in the approximate computation of the auxiliary $\eta$'s in (\ref{fixptit}). We can observe that by increasing the number of subspaces and inner iterations we can significantly improve the rate of convergence with respect to the external iterations.}}\label{figure5}
\end{figure}
\end{center}

\section{Conclusion} 
Optimization of functionals promoting sparse recovery, e.g., $\ell_1$-minimization and total variation minimization (where the sparsity is at the level of derivatives), were proposed in order to extract few significant features of the solution originally defined in very high dimensions. As a matter of fact, these minimizations cannot be performed by ordinary methods when the dimension scale is extremely large, for speed, resources, and memory restrictions.  Hence, domain decomposition or subspace correction methods have to be invoked in these cases. Our work contributes to remedy the lack of such methods for these specific problems. We introduced parallel and alternating optimization algorithms on sequences of orthogonal subspaces of a Hilbert space, for the minimization of energy functionals  involving convex constraints coinciding with semi-norms for a subspace. We provided an efficient numerical method for the implementation of the algorithm via oblique thresholding, defined by suitable Lagrange multipliers. 
 It is important to notice that, on the one hand, these algorithms are realized by re-utilizing the basic building blocks of standard proximity-map iterations, e.g., projections onto convex sets, no significant complications in the implementations occur. On the other hand, several tricks can be applied in order to limit the computational load produced by the multiple iterations occurring on several subspaces (compare subsection \ref{tricks}).
We investigated the convergence properties of the algorithms, providing sufficient conditions for ensuring the convergence to minimizers.
We showed the applicability of these algorithms in delicate situations, like in domain decomposition methods for singular elliptic PDE's with discontinuous solutions in 1D and 2D, and in accelerations of $\ell_1$-minimizations. The numerical experiments nicely confirm the results predicted by the theory. 

\section*{Acknowledgments}

The authors thank Peter A. Markowich and the Applied Partial Differential Equations Research Group of the Department of Applied Mathematics
and Theoretical Physics, Cambridge University, for the hospitality and the fruitful discussions during the late preparation of this work.
M. Fornasier thanks Ingrid Daubechies for the intensive discussions on sparse recovery and the Program in Applied and Computational Mathematics, Princeton University, for the hospitality, during the early preparation of this work. M. Fornasier acknowledges the financial support provided by the European Union's Human Potential Programme under contract MOIF-CT-2006-039438. 
C.-B. Sch\"onlieb acknowledges the financial support provided by
the Wissenschaftskolleg (Graduiertenkolleg, Ph.D. program) of the Faculty
for Mathematics at the University of Vienna, supported by the Austrian
Science Fund.
The results of the paper also contribute to the project WWTF Five senses-Call 2006, Mathematical Methods for Image Analysis and Processing in the Visual Arts.

\providecommand{\bysame}{\leavevmode\hbox to3em{\hrulefill}\thinspace}
\providecommand{\MR}{\relax\ifhmode\unskip\space\fi MR }
\providecommand{\MRhref}[2]{%
  \href{http://www.ams.org/mathscinet-getitem?mr=#1}{#2}
}
\providecommand{\href}[2]{#2}

\end{document}